\UseRawInputEncoding

\documentclass[12pt]{amsart}
\usepackage{epsfig,color, amsmath, amssymb}
\usepackage{esint}
\usepackage{hyperref, mathrsfs}
\usepackage{xcolor}
\usepackage{bm}
\usepackage{enumitem}
\usepackage{mathtools}
\usepackage{comment}
  \hypersetup{colorlinks}
\hypersetup{citecolor=blue}
\hypersetup{urlcolor=blue}

\makeatother

\theoremstyle{definition}

\def\fnum{equation} 
\newtheorem{Thm}[\fnum]{Theorem}
\newtheorem{thm}[\fnum]{Theorem}

\newtheorem{cor}[\fnum]{Corollary}
\newtheorem{Que}[\fnum]{Question}
\newtheorem{Lem}[\fnum]{Lemma}
\newtheorem{lem}[\fnum]{Lemma}

\newtheorem{defn}[\fnum]{Definition}

\newtheorem{Rem}[\fnum]{Remark}
\newtheorem{rem}[\fnum]{Remark}
\newtheorem{Pro}[\fnum]{Proposition}
\newtheorem{prop}[\fnum]{Proposition}

\numberwithin{equation}{section}
\newcommand{\gexp}{{\text {gexp}}}

\newcommand{\diam}{{\text {diam}}}

 \newcommand{\N}{\ensuremath{\mathbb{N}}}
 
  \newcommand{\R}{\ensuremath{\mathbb{R}}}
  
 \newcommand{\Z}{\ensuremath{\mathbb{Z}}}
 
 \newcommand{\RP}{\ensuremath{\mathbb{RP}}}

 \newcommand{\ba}{\begin{align*}}
 \newcommand{\ea}{\end{align*}}

\newcommand{\Rad}{\text{Rad}}

\newcommand{\Id}{\text{Id}}

\def\RP{{\bold{RP}}}
\def\SS{{\mathbb S}}

\newcommand{\co}{{\text {c}}}

\newcommand{\kang}{\tilde{\measuredangle}^{\kappa}}
\newcommand{\kgexp}{\text{gexp}^{\kappa}_p}

\newcommand{\eps}{\varepsilon}

\DeclareMathOperator{\RCD}{RCD}

\newcommand{\dd}{\mathop{}\!\mathrm{d}}
\newcommand{\mm}{\mathfrak{m}}

\newcommand{\kcone}{C_{\kappa}}

\newcommand{\curv}{\operatorname{curv}}

\DeclareMathOperator{\vol}{vol}
\DeclareMathOperator{\itan}{itan}
\DeclareMathOperator{\itank}{itan_\kappa}
\DeclareMathOperator{\tank}{\tan_\kappa}
\DeclareMathOperator{\snk}{sn_\kappa}
\DeclareMathOperator{\csk}{cs_\kappa}
\DeclareMathOperator{\Intr}{Int}

\newcommand{\pik}{\pi_\kappa}
\def\co{\colon\thinspace}
\DeclareMathOperator\Int{Int}

\begin{document}
\title{Sharp bounds and rigidity for volumes of boundaries of Alexandrov spaces}

\author{
Qin Deng}
\address{Qin Deng, Massachusetts Institute of Technology, \url{qindeng@mit.edu}}
\author{
 Vitali Kapovitch
 } \thanks{V.K. is partially supported by a Discovery grant from NSERC }
\address {Vitali Kapovitch, University of Toronto, \url{vtk@math.toronto.edu}} 
\maketitle

\begin{abstract}
We obtain sharp volume bounds on the boundaries of Alexandrov spaces with given lower curvature bound, dimension, and radius. We also completely classify the rigidity case and analyze almost rigidity. Our results are new even for smooth manifolds with boundary.
\end{abstract}

\tableofcontents

\section{Introduction}

Let $(X^m, \dd)$ be an Alexandrov space with curvature bounded below by some $\kappa \in \R$ and dimension $m$.  Recall that the radius $\Rad(X)$ of $X$ is defined as
\begin{equation*}
    \Rad(X) = \inf_{x \in X} \sup_{y \in X} \dd(x,y).
\end{equation*}
We denote by $M^{m}(\kappa)$ the $m$-dimensional model space with constant curvature $\kappa$ and $\pik$ the diameter of $M^m(\kappa)$, i.e., $\pik = \pi/\sqrt{\kappa}$ if $\kappa > 0$ and $\infty$ otherwise.

Our first result is a sharp upper bound on the volume of the boundary of $X$ of radius $R$.
\begin{Thm}\label{thm-vol-bound}
Let $X^m$ be Alexandrov of $\curv\ge\kappa$. Let $R > 0$ and assume in addition that $R \leq \pik/2$ if $\kappa > 0$. If $\overline{B_R(p)}=X$ for some $p \in X$ then 
\begin{equation}
    \mathcal{H}^{m-1}(\partial X) \le \mathcal{H}^{m-1}(\partial B_{R}^{\kappa}),
\end{equation}
where $B_{R}^{\kappa}$ is the ball of radius $R$ in $M^{m}(\kappa)$ (we drop the dependence on $m$ for simplicity since we will fix dimension).
Furthermore, if $\overline{B_R(p)}=X$  for some  $p\in\partial X$ then

\begin{equation}\label{eq-upper-br}
    \mathcal{H}^{m-1}(\partial X) \le \mathcal{H}^{m-1}(\partial B_{R,+}^{\kappa} ),
\end{equation}
where $B_{R,+}^{\kappa}$ is the  half-ball  of radius $R$ in  $M^{m}(\kappa)$, i.e., it is the ball of radius $R$ around a boundary point in  the half space $M^{m}_+(\kappa)$ in  $M^{m}(\kappa)$.
\end{Thm}

\begin{Rem}
Note that  $\mathcal{H}^{m-1}(\partial B_{R,+}^{\kappa}) < \mathcal{H}^{m-1}(\partial B_{R}^{\kappa}) $ unless $\kappa>0$ and $R=\pik/2$ in which case they are equal. Hence, except for that case the upper bound given by Theorem \ref{thm-vol-bound} is stronger if $p\in \partial X$.

\end{Rem}
\begin{Rem}The bound given by Theorem~\ref{thm-vol-bound} is new even in the smooth setting, i.e., when $X=(M^m,g)$ is a compact smooth Riemannian manifold of $\sec\ge \kappa$ with locally convex boundary.
\end{Rem}

In the special case $\kappa>0, R=\pik/2$, the estimate given by Theorem \ref{thm-vol-bound} follows from Petrunin's solution to Lytchak's problem, who proved in \cite{petrun-semiconcave}[Section 3.3.5]  that if $X^m$ is Alexandrov of $\curv\ge 1$ then $\mathcal{H}^{m-1}(\partial X)\le \vol_{m-1} \SS^{m-1}$. Note that if $X^m$ is Alexandrov of $\curv\ge 1$ with nonempty boundary and $p\in X$ is the soul of $X$ then $X=B_{\pi/2}(p)$, see \cite{petrun-semiconcave}[Section 3.3.5] or \cite[Section 6.3]{Per1}. Therefore Petrunin's result gives the sharp upper bound in Theorem~\ref{thm-vol-bound} for $\kappa>0, R=\pik/2$. However, Petrunin's proof does not generalize to other values of $R$  and $\kappa$.

Next, let us note that an implicit bound on the volume of the boundary in this theorem follows from a result of Fujioka \cite[Theorem 6.5] {Fuj-uniform-bounds-extr}, where an upper bound is obtained on the volume of any extremal subset of $X$ in terms of dimension, lower curvature and upper diameter bounds of $X$.  Fujioka's result also gives an implicit local bound under the assumptions of Theorem \ref{thm-vol-bound-local} below.

As was pointed out to the authors by Daniele Semola, an implicit upper bound on $\mathcal{H}^{m-1}(\partial X)$ under the assumptions of Theorems \ref{thm-vol-bound} and \ref{thm-vol-bound-local} also follows from work of Li and Naber \cite{Li-Naber}. Since this is not mentioned explicitly in their paper let us elaborate on this point. In \cite{Li-Naber}[Corollary 1.4] it is shown that for any integer $m$ and  $\eps>0,\kappa\in \R$, there is $C=C(m,\eps,\kappa)>0$ such that if $X^m$ Alexandrov space of $\curv\ge \kappa$ and $p\in X$ then

\begin{equation}\label{eq-li-naber}
\mathcal H^k(\mathcal S_\eps^k(X)\cap B_1(p))\le C.
\end{equation}
Here $\mathcal S_\eps^k(X)$ stands for the $(k, \eps)$ quantitative stratified set. See \cite{Li-Naber} for the definition. For any small $\eps<\eps_0(m,\kappa)$  it is immediate from the definition that if $k=m-1$ and $p\notin \mathcal S_\eps^{m-1}(X)$ then $p$ is $(m,\eps)$-strained and hence it cannot be a boundary point by \cite{BGP}. Therefore $\partial X\subset \mathcal S_\eps^{m-1}(X)$ for any $\eps<\eps_0$ and so \eqref{eq-li-naber} gives a uniform bound  $\mathcal H^{m-1}(\partial X\cap B_1(p))\le C(m, \kappa)$. See also \cite{BNS22} for implicit bounds in the setting of spaces with Ricci curvature bounded below under a further quantitatively noncollapsed assumption. 

 More generally we prove the following bound which is also sharp.
 \begin{Thm}\label{thm-tancone-vol-bound}
 Let $X^m$ be Alexandrov of $\curv\ge\kappa$. Let $R > 0$ and assume in addition that $R < \pik/2$ if $\kappa > 0$. If $\overline{B_R(p)}=X$ for some $p \in X$ then
 \begin{equation}\label{thm-vol-bound-2}
    \mathcal{H}^{m-1}(\partial X) \le \mathcal{H}^{m-1}(\partial B_{R}^{\kappa}(o_p)),
\end{equation}
where $B_{R}^{\kappa}(o_p)$ is the ball of radius $R$ in $T_pX$ equipped with the $\kappa$-cone metric (see section \ref{sec:prelim} for the definition).
 \end{Thm}

The upper bound in Theorem~\ref{thm-vol-bound} holds locally too. 

\begin{Thm}\label{thm-vol-bound-local}
Let $X^m$ be Alexandrov with $\curv\ge \kappa$. Let $R > 0$ and assume in addition that $R \leq \pik/2$ if $\kappa > 0$. Then for any $p\in X$ it holds that
\begin{equation}
    \mathcal{H}^{m-1}(\partial X\cap B_R(p)) \le \mathcal{H}^{m-1}(\partial B_{R}^{\kappa}),
\end{equation}
where $B_{R}^{\kappa}$ is the ball of radius $R$ in $M^{m}(\kappa)$.

If $p\in\partial X$ then
\begin{equation}
    \mathcal{H}^{m-1}(\partial X\cap B_R(p)) \le \mathcal{H}^{m-1}(\partial B_{R,+}^{\kappa} ),
\end{equation}
where $B_{R,+}^{\kappa}$ is the  half-ball  of radius $R$ in  $M^{m}(\kappa)$.
Moreover, in either of these cases the inequalities are strict unless $X=\overline{B_R(p)}$ and $X$ is one of the spaces given by Theorem~\ref{main thm}.
\end{Thm}

%In this paper we will be concerned with $X$ of radius $R = \Rad(X) < \infty$ with boundary which satisfies the following maximal boundary volume condition

%Note that in the case where $\kappa = 1$, it is known that Alexandrov spaces with boundary must have radius at most $\pi/2$. Indeed, if $\Rad(X) > \pi/2$ then it is known that $X$ must be homeomorphic to the sphere, c.f. \cite{GP93}, \cite{petrun-semiconcave}[Corollary 5.5.2]. The main result of this paper is the following rigidity theorem.
We also completely analyze the equality case in the setting of Theorem~\ref{thm-vol-bound}. For $\kappa > 0$, we will refer to the intersection of two closed hemispheres intersecting at an angle of $\alpha$ in $M^{m}(\kappa)$ as an Alexandrov lens and denote it by $L^{m,\kappa}_{\alpha}$ after the convention of \cite{GP-almost-maximal}. Note that $L^{m, \kappa}_{\pi}$ is just the closed hemisphere in $M^{m}(\kappa)$.
\begin{Thm}\label{main thm}
    Let $X^m$ be Alexandrov of $\curv\ge\kappa$. Let $R > 0$ and assume that $\overline{B_{R}(p)}=X$ for some $p \in X$. If $X$ satisfies
    \begin{equation}\label{max bd vol}
    \mathcal{H}^{m-1}(\partial X) = \mathcal{H}^{m-1}(\partial B_{R}^{\kappa})
\end{equation}
then $X = \overline{B_R^{\kappa}} \subset M^{m}(\kappa)$ if either $\kappa \leq 0$ or if $\kappa > 0$ and $R < \pik/2.$ In the case that $\kappa > 0$ and $R = \pik/2,$ $X$ must be $L^{m, \kappa}_{\alpha}$ for some $\alpha \in (0, \pi]$.

Furthermore, if $p\in\partial X$ and $X$ satisfies 
\begin{equation}\label{eq-upper-br}
    \mathcal{H}^{m-1}(\partial X) = \mathcal{H}^{m-1}(\partial B_{R,+}^{\kappa} ),
\end{equation}
then $X = \overline{B_{R,+}^{\kappa}} \subset M^{m}(\kappa)$ if either $\kappa \leq 0$ or if $\kappa > 0$ and $R < \pik/2.$ In the case that $\kappa > 0$ and $R = \pik/2,$ $X$ must be $L^{m, \kappa}_{\alpha}$ for some $\alpha \in (0, \pi/2]$. 
\end{Thm}
Rigidity in the case $\kappa = 1, R = \pi/2$ was first addressed and proved in \cite{GP-almost-maximal}. However, their proof strongly uses the rigidity from the maximality of radius and does not generalize to arbitrary $\kappa$ and $R$. Instead we use a much more analytical approach. We will use their rigidity result in one step of our argument but it should also be possible to use our methods to give an alternative proof for this case directly using induction. For the sake of brevity, we will not write out the details but will give a short discussion at the relevant section (see Remark \ref{alternate proof}).  

We also analyze almost rigidity.

\begin{Thm}\label{thm-almost-rigidity}
For any $\epsilon > 0$, there exists $\delta(\kappa, m, R)$ such that the following holds:
Let $X^m$ be Alexandrov with $\curv \ge \kappa$. Let $R > 0$ and assume that $\overline{B_{R}(p)}=X$ for some $p \in X$. Assume $X$ satisfies
\begin{equation}\label{almost max bd vol}
  \mathcal{H}^{m-1}(\partial X) > \mathcal{H}^{m-1}(\partial B_{R}^{\kappa}) - \delta.
\end{equation}
If $\kappa > 0$ and $R < \pik/2$ or $\kappa \leq 0$ then
\begin{equation*}
	\dd_{GH}(X, \overline{B_R^{\kappa}}) < \epsilon.
\end{equation*}
If $\kappa > 0$ and $R = \pik/2$, then there exists $Y$, which is either $L^{m, \kappa}_{\alpha}$ for some $\alpha \in (0, \pi]$ or the closed ball of radius $R$ in $M^{m-1}(\kappa)$, so that
\begin{equation*}
	\dd_{GH}(X, Y) < \epsilon.
\end{equation*} 
Furthermore, assuming that $p \in \partial X$ and that $X$ satisfies 
\begin{equation}\label{almost max bd vol 2}
  \mathcal{H}^{m-1}(\partial X) > \mathcal{H}^{m-1}(\partial B_{R, +}^{\kappa}) - \delta.
\end{equation}
If $\kappa > 0$ and $R < \pik/2$ or $\kappa \leq 0$ then
\begin{equation*}
	\dd_{GH}(X, \overline{B_{R,+}^{\kappa}}) < \epsilon.
\end{equation*}
If $\kappa > 0$ and $R = \pik/2$, then there exists $Y$, which is either $L^{m, \kappa}_{\alpha}$ for some $\alpha \in (0, \pi/2]$ or the closed ball of radius $R$ in $M^{m-1}(\kappa)$, so that
\begin{equation*}
	\dd_{GH}(X, Y) < \epsilon.
\end{equation*} 
\end{Thm}
We comment briefly on the proof. In the case where $\kappa > 0$ and $R < \pik/2$ or $\kappa \leq 0$, we rule out the possibility of collapse in the almost rigidity situation and so the previous theorem follows directly from standard Gromov-Hausdorff compactness along with rigidity. In the case where $\kappa > 0$ and $R = \pik/2$, collapsing is clearly possible and so a more delicate analysis of collapsing limits is required to show that a collapse to the closed ball of radius $\pik/2$ in $M^{m-1}(\kappa)$ is all that can happen. 

\subsection{Open questions}

We end this section with some open questions. The most natural one is what can be said if lower sectional curvature bound is replaced by lower Ricci curvature bound. See \cite{DG18, KM21, BNS22} and references therein for the theory of the boundary of noncollapsed $\RCD(K,N)$ spaces.
\begin{Que}\label{que-ric-gen}
Let $(X, \dd, \mathcal{H}^N)$ be a noncollapsed $\RCD(K,N)$ space. Suppose that $\Rad (X) = R$. Is there an upper bound on $\mathcal{H}^{m-1}(\partial X)$? If so is the sharp bound given by the volume of the boundary of an $R$ ball in model space? Can one classify the rigidity cases?
\end{Que}

The current best estimate of this type is given by \cite{BNS22}[Theorem 1.4], where they prove an implicit bound under a further quantitative noncollapsed assumption. More precisely, they prove the following.
\begin{Thm}\label{thm-BNS-imp-bound}
Let $(X, \dd, \mathcal{H}^N)$ be a noncollapsed $\RCD(K, N)$ space. Let $p \in X$ so that $\mathcal{H}^N(B_1(p)) > v > 0$. Then for any $x \in \partial X \cap B_1(p)$ and $r \in (0,1)$,
\[
\mathcal{H}^{N-1}(B_r(x) \cap \partial X) \leq C(K, N, v)r^{N-1}.
\]
\end{Thm}
However, it remains an open question whether the dependence on $v$ can be removed and so even the existence of an implicit bound in Question \ref{que-ric-gen} is not known.

In the case where $X$ is Alexandrov with $\curv \ge \kappa > 0$ and with nonempty  boundary, it is natural to assume that $\Rad(X) = R \le \pik/2$. This is because if $R > \pik/2$ then $\partial X = \varnothing$ due to the Grove-Petersen radius sphere theorem \cite{GrPe-sphere-thm}. For $\RCD(\kappa(N-1),N)$ spaces with $\kappa > 0$ we do not know if there is an example with $R >  \pik/2$ and $\partial X \neq \varnothing$. 
\begin{Que}
Let $(X, \dd, \mathcal{H}^N)$ be a noncollapsed $\RCD(N-1, N)$ space with nonempty boundary. Is it possible that $\Rad(X) > \pi/2$?
\end{Que}

A trivial but useful corollary of Theorem \ref{thm-vol-bound-local} (see the proof of Theorem \ref{thm-almost-rigid-2}) is that a uniform lower bound on the volume of the boundary along a convergent sequence of Alexandrov spaces bounds the dimension of collapse. More precisely, if we have a sequence of Alexandrov spaces $(X^m_i, \dd_i, p_i)$ with $\curv \ge \kappa$ converging to $(X_\infty, \dd_\infty, p_\infty)$ and it holds that $\mathcal{H}^{m-1}(\partial X_i \cap B_1(p_i)) > c > 0$ for all $i$, then $\dim(X_\infty) \ge m-1$. More generally, a uniform lower bound on $\mathcal{H}^{k}(\mathcal{S}^k_{\epsilon}(X_i) \cap B_1(p_i))$ would also bound the dimension of collapse using the implicit bounds of \cite{Li-Naber}. In the $\RCD$ case, a similar result would follow easily in the same way if the implicit bound of Theorem \ref{thm-BNS-imp-bound} can be given without dependence on $v$. Since this is not known, we have the following question.

\begin{Que}\label{que-collapse}
If $(X_i, \dd_i, \mathcal{H}^N, p_i)$ are noncollapsed $\RCD(K,N)$ spaces converging to some $(X_\infty, \dd_\infty, \mm_\infty, p_\infty)$ satisfying $\mathcal{H}^{N-1}(\partial X_i \cap B_1(p_i)) > c > 0$ for all $i$. Is it true that $\dim(X_\infty) \ge N-1$? 
\end{Que} 
This question can be asked for either $\dim(X_\infty)$ being the Hausdorff dimension or the essential dimension of $X_\infty$. If a counterexample can be given for Question \ref{que-collapse} for the Hausdorff dimension then it would immediately imply that the estimates of Theorem \ref{thm-BNS-imp-bound} cannot be made independent of $v$.

Finally, in the setting of Alexandrov spaces one can ask for optimal bounds on other geometric invariants of the boundary.
\begin{Que}
Let $(X^m, \dd)$ be Alexandrov with $\curv \ge \kappa$ and $\Rad(X) = R$. What are the sharp bounds on the diameter and radius of a connected component of $\partial X$? What are the rigidity cases?
\end{Que}
An implicit bound can be given using \cite{Fuj-uniform-bounds-extr}[Theorem 6.6], which more generally bounds the maximal number of $\epsilon$-separated points on an extremal set with respect to its intrinsic geometry. Of course, sharp bounds would follow immediately in the case $\kappa = 1$ and $R = \pi/2$ from the conjecture that $\partial X$ is Alexandrov. We mention that the diameter rigidity of the boundary for this case was analyzed in \cite{GL21} under the additional assumption that $\partial X$ is a smooth Riemannian manifold of dimension $m-1$ with $\sec \ge 1$.\subsection{Acknowledgements}  The authors are grateful to Elia Bru\'{e}, Alexander Lytchak, Keaton Naff and Daniele Semola for helpful discussions and comments.  
\section{Preliminaries}\label{sec:prelim}

We begin by introducing some definitions and notations, see \cite{BGP,AKPbook} for a more careful treatment.

Given an Alexandrov space $X$ with curvature bounded below by $1$, we denote by $\kcone(X)$ the standard warped cone with curvature bounded below by $\kappa$. Notice that in the case of $\kappa > 0$ the cone is really a suspension where the radial parameter is between $0$ and $\pik$. We will usually write $C(X)$ for $C_0(X)$ since it is the standard Euclidean cone. 

Fix now dimension $m \in \N$ and $\kappa > 0$. Let $(X, \dd)$ be an Alexandrov space of dimension $m$ and curvature lower bound $\kappa$. For each $p \in X$ we denote by $\Sigma_pX$ the space of directions at $p$ and $(T_pX, o_p)=(C(\Sigma_pX), o_p)$ to be the tangent cone at $p$. When it is unambiguous we will often drop $X$ in above notations and just write $\Sigma_p$ and $T_p$.

Given $p, q \in X$, we will use $[pq]$ to denote a (usually not unique) geodesic between $p$ and $q$. We will use $\log_p([pq]) \in T_pX$ to denote the vector corresponding to this geodesic, i.e., $\log_p([pq])$ is in the direction of $[pq]$ and has magnitude $\dd(p,q)$. 

On the closed ball $\overline{B_{\pik/2}(o_p)} \subset T_p$, one may define a gradient exponentiation map $\gexp^{\kappa}_p: \overline{B_{\pik/2}(o_p)} \to X$ which shares similar properties as the usual exponentiation map on $M^{m}(\kappa)$. The original construction is from \cite{Petr-Per-grad-curves}. We point out that the map depends on $\kappa$.
The for $v\in \Sigma_p$ the curve $\sigma(t)=\kgexp(tv)$ satisfies the IVP

\begin{equation}\label{gexp-defn}
\begin{cases}
\sigma'_+(t)=\frac{\tank(\dd(\sigma(t), p))}{\tank(t)}\nabla_{\sigma(t)}\dd(\cdot, p)\\
\sigma'(0)_+=v
\end{cases}
\end{equation}
where $\tank := \snk/\csk$ is the generalized tangent function, see \cite{AKPbook}[Section 1.A] for the definitions of $\snk$ and $\csk$. 

Further, $\sigma(t)$ is a reparameterized gradient curve of the modified distance function $f(\cdot)=\itank(\dd(\cdot, p)),$ where $\itan_\kappa(t):=\int_0^t \tank(s)ds$. The function $f$ is semiconcave on $B_R(p)$ by construction. We refer to the proof of \cite{AKPbook}[16.31] for a careful discussion of these facts.

 We outline some of the key properties of the gradient exponentiation map we shall need below and refer to \cite[Chapter 16]{AKPbook} and \cite[Chapter 3]{petrun-semiconcave} for the proofs. The first proposition says that that $\kgexp$ agrees with shortest  geodesics where they exist.
\begin{Pro}
    For any $q \in X$ and any geodesic $[pq]$ from $p$ to $q$, $\kgexp(\log_p([pq])) = q$.
\end{Pro}
To state the next proposition we introduce the notation for comparison angles in $M^2(\kappa)$. For $a, b, c >0,$ we define
$\tilde{\measuredangle}^{\kappa}(a; b, c)$ to be the angle corresponding to the triangle with side lengths a, b, and c in $M^2(\kappa)$ that is opposite of the side with length $a$. For obvious reason we require $a+b+c < 2\pik.$ In the case where the triangle inequality fails, for example if $b+a < c$ or if $c+a < b$, we will extend the usual definition so that $\tilde{\measuredangle}^{\kappa}(a; b, c) = 0$. Similarly, $\tilde{\measuredangle}^{\kappa}(a; b, c) = \pi$ if $b+c < a$. The following proposition says that the gradient exponentiation map satisfies similar comparison estimates as the exponentiation map on $M^{m}(\kappa)$ up until a radius of $\pik/2.$
\begin{Pro}\label{exp comp}
    Let $q_1, q_2 \in X$ so that $\dd(p, q_i) \leq \pik/2$ for each $i$ and let $v_1, v_2 \in T_p$ be vectors which correspond to some geodesic from $p$ to $q_1$ and $q_2$ respectively. Let $s_1, s_2 \geq 1$ so that $|s_iv_i| \leq \pik/2$ for each $i$, then
    \begin{equation*}
        \kang(\dd(\kgexp(s_1v), \kgexp(s_2v)); |s_1v_1|, |s_2v_2|) \leq \kang(\dd(q_1, q_2); |v_1|, |v_2|).
    \end{equation*}
\end{Pro}
If we identify $B_{\pik/2}(o_p) \subset T_p$ with $B^\kappa_{\pik/2}(o_p) \subset \kcone(\Sigma_p)$ in the obvious way and use this to define an exponential map on the latter, which by abuse of notation we also denote $\kgexp$, then from the previous proposition we immediately have the following.
\begin{Pro}\label{exp short}
    The map $\kgexp: B^{\kappa}_{\pik/2}(o_p) \to X$ is $1$-Lipschitz. 
\end{Pro}

\section{Upper bound on the volume of the boundary}
In this section we prove Theorems ~\ref{thm-vol-bound}, ~\ref{thm-tancone-vol-bound}, and ~\ref{thm-vol-bound-local}. We will focus on the first two theorems and give a brief discussion for the last one at the end of the section.

The main step is the following key Lemma which by Proposition \ref{exp short} immediately implies the above theorems. In this lemma we only assume that $X^m$ is Alexandrov with $\curv \geq \kappa$ and $X = \overline{B_R(p)}$ for some $p \in X$. 

\begin{Lem}\label{lem-exp-onto-bry}
   The boundary $\partial X$ is contained in $\kgexp(\partial B_{R}(o_p))$. 
\end{Lem}
\begin{Rem}
The lemma generalizes the same statement for $\kappa=1$ and $R=\pi/2$, which was used by Petrunin to solve Lytchak's problem in \cite{petrun-semiconcave}[Proof of 3.3.5].
However, Petrunin's proof does not generalize to general $\kappa$ and $R$ and to the local version needed for Theorem ~\ref{thm-vol-bound-local}.  We give a different argument that does work in all of these cases.
\end{Rem}

\begin{proof}
First, recall that for any $v\in \Sigma_p$, the curve $t\mapsto \kgexp(tv)$ for $0\le t\le R$ is a reparameterized gradient curve of the modified distance function $f(\cdot)=\itank(d(\cdot, p)),$ The function $f$ is semiconcave on $B_R(p)$ by construction.

%The reparameterization function $\alpha(t)$  satisfies the condition that $\beta=(\alpha)^{-1}$ solves the ODE $\beta'=\tank(\beta)$, i.e it's a gradient curve of $\itank(t)$ on $\R$.

We will consider two separate cases, depending on whether or not $p\in \partial X$.

{\bf Case 1} Suppose $p\notin\partial X$.
%has derivative between two positive constants $C_1<\alpha'(t)<C_2$  on any interval $[0, D]$ where $0<D<\omega^\kappa/2$.
Let $x\in \partial X$. We will show $x \in \kgexp(\partial B_{R}(o_p))$. Let $v\in \Sigma_pX$ be the initial direction of a shortest geodesic $[px]$.
 
If $x$ is a critical point of $\dd(\cdot, p)$ (and hence of $f$), then by the definition of $\kgexp$ we have that $x=\kgexp(R\cdot v)$ and hence $x\in \gexp_p(S_R(o_p))$, where $S_R(o_p)$ is the sphere of radius $R$ in $T_pX$. Notice that in this case $T_pX$ is without boundary so $S_R(o_p)$ is exactly $\partial B_{R}(o_p)$. 
 
Suppose now that $x$ is not a critical point of  $f$ and $\dd(\cdot, p)$. Note that this means that $\dd(p,x)<R$, though in any case what we would like to prove for $x$ is obvious if $\dd(p,x)=R$.
 
Since $\partial X$ is extremal, by \cite{Fuj-uniform-bounds-extr} (cf. \cite{petrun-semiconcave}[Section 2.2, Property 3]) there exists a gradient curve of $f$, which we denote $\gamma:  (-\infty, t_{\max} )\to \partial X$, such that $\gamma(0)=x$. If $\kappa \le 0$ then we may take $t_{\max} = \infty$. If $\kappa > 0$ and $t_{\max} < \infty$, then $\lim_{t\to t_{\max}} \dd(p,\gamma(t))=\pik/2$. In particular, $t_{\max}=\infty$ if $R<\pik/2$. Note that there is a small $\eps>0$ such that $\gamma(t)$ is not a critical point of $f$ for any $t\le \eps$ and so $\gamma$ is a topological embedding of any finite interval contained in $(-\infty, \eps]$.

Fix any $t<0$ and let $q_t=\gamma(t)$. Let $w_t$ be the initial direction of a shortest geodesic $[p,q_t]$ and let $d_t=\dd(p,q_t)$. Then $\kgexp(d_t\cdot w_t)=q_t$ and for any $d_t \le s \le \pik/2$ we have that 
\[
\kgexp(s\cdot w_t)=\gamma (\tau(t, s))
\]
for some continuous function $\tau(t, s)$ defined on $\{(t,s): t < 0, d_t \le s \le \pik/2\}$. Clearly, $\tau(t, d_t)=t$ and $\tau(t, s)$ is increasing in $s$.
 
We claim that there is a $t<0$ such that $\tau(t,s)<0$ for any $d_t\le s\le R$.
Indeed we have that both $\dd(\cdot, p)$ and $f$ are strictly increasing along $\gamma(t)$ for $-\infty<t<\eps$.  %They also have uniformly bounded gradients on this interval. 
Since $p\notin \partial X$ and $\gamma\subset \partial X$ we have that $\dd(\gamma(t), p)$ is bounded below by a positive constant.

We have $f(\gamma(t))'_+=|\nabla_{\gamma(t)}f|^2=\tank(\dd(\gamma(t),p))^2 |\nabla_{\gamma(t)}\dd(\cdot, p)|^2\ge c  |\nabla_{\gamma(t)}\dd(\cdot, p)|^2$. Therefore for sufficiently negative $t$ we have that $\int_t^{t+R}  |\nabla_{\gamma(s)}\dd(\cdot, p)|^2ds$ can be made arbitrary small due to integrability reasons. By Cauchy-Scwartz the same is true for $\int_t^{t+R}  |\nabla_{\gamma(s)}\dd(\cdot, p)|ds$.  Now the speed of the radial curve $s\to \kgexp(sw_t)$ is bounded by $|\nabla \dd(\cdot, p)|$ at every point by \eqref{gexp-defn}, so we get that the length of this curve on the interval $[d_t,R]$ can be made arbitrary small for sufficiently negative $t$. By Lipschitz properties of $f$ this implies that $f(\kgexp(s\cdot w_t))<f(\gamma(0))$ for any $s\in [d_t, R]$ for sufficiently negative $t$. This means that by choosing sufficiently negative $t$ we can ensure that $\tau(t,R) < 0$. We fix such a $t_0$.

Consider now the map $\phi\co [t_0, 0]\to \R$ given by $\phi(t)=f(\gamma ( \tau(t, R)))=f(\kgexp(R\cdot w_t))$. This map is continuous with $\phi(t_0)<f(\gamma(0))$ and $\phi(0)>f(\gamma(0))$. Hence by the Intermediate Value theorem there is a $t\in[t_0,0]$ such that $\phi(t)=f(\gamma(0))$. Since $f$ is strictly increasing along $\gamma$ for $t \le 0$ it follows that $x = \gamma(0) = \kgexp(R\cdot w_t)) \in \kgexp(S_R(0))$, as claimed. 

This finishes the proof in Case 1.

{\bf Case 2} Now suppose that $p\in\partial X$. 
We have that $\partial B_{R}^{\kappa}(o_p)$ consists of two parts: the sphere $S_R(o_p)$ and the closed  ball  of radius $R$ in $T_p\partial X$ which is the $ \kappa$-cone over $\partial \Sigma_p=\Sigma_p\partial X$.
Let us denote this ball by $\bar B_R^{\partial}(o_p)$.

We need to show that $\partial X\subset \kgexp(\bar B_R^{\partial}(o_p)\cup S_R(o_p))$.

Suppose $x\in \partial X$ but $x\notin \kgexp(\bar B_R^{\partial}(o_p))$. If $\dd(x,p)=R$ the statement of the lemma is obvious so assume $d(p,x)<R \leq \pik/2$. 

As in the proof of Case 1 we first construct a gradient curve of $f$, denoted $\gamma:  (-\infty, t_{\max} )\to \partial X$  such that $\gamma(0)=x$.  Then $f$ is strictly increasing on  $(-\infty, \eps)$ and none of the points on this curve are critical for $f$.

There are two possibilities to consider. If $\lim\limits_{t\to-\infty} f(\gamma(t))>0$ then this curve stays a definite distance away from $p$ and the same proof as in Case 1 works.

Now suppose that $\lim\limits_{t\to-\infty} f(\gamma(t))=0$. This means that $\dd(\gamma(t), p)\to 0$ as $t\to-\infty$.

As before for any $t<0$  let $q_t=\gamma(t)$, $w_t$ be the initial direction of a shortest geodesic $[p,q_t]$ and $d_t:=\dd(p,q_t)$. Then $\kgexp(d_t\cdot w_t)=q_t$ and for any $d_t \le s \le \pik/2$ we have that 

\[
\kgexp(s\cdot w_t)=\gamma (\tau(t, s))\]
for some continuous $\tau$. 

The vectors $w_t$ subconverge to some $w\in \Sigma_p\partial X$ as $t\to -\infty$ and hence $\kgexp(s\cdot w_t)=\gamma (\tau(t, s))\to \kgexp(s\cdot w)$  as $t\to-\infty$ for all $s>0$.  In particular it holds for $s\le R$.

But $\kgexp(s\cdot w)\in      \kgexp(\bar B_R^{\partial}(o_p))$ for $s\le R$ and  hence none of these points can be equal to $x$. Hence $\kgexp(R\cdot w)=\gamma(T)$ for some $T<0$. Now the same Intermediate Value Theorem argument as in Case 1 applies and shows that $x\in \kgexp(S_R(o_p))$. This concludes the proof in Case 2.

\end{proof}

Since $\kgexp$ is $1$-Lipschitz on $B_{R}^{\kappa}(o_p)$ by Proposition \ref{exp short}, it is nonincreasing on $\mathcal{H}^{m-1}$ measure. This along with the previous lemma immediately implies Theorem~\ref{thm-tancone-vol-bound}. Using the maximal volume bound for $(m-1)$-dimensional Alexandrov spaces with $\curv \geq 1$ with and without boundary, Theorem~\ref{thm-tancone-vol-bound} implies Theorem ~\ref{thm-vol-bound}. 

Observe now that in the above proof $\dd(\cdot, p)$ is nondecreasing along $\gamma$. Hence if we only assume that $\gamma(0)\in \partial X\cap B_R(p)$ (without assuming $X = B_R(p)$) then we still have $\gamma(t)\in \partial X\cap B_R(p)$ for any $t\le 0$. With that observation the same proof as in the lemma gives that $\partial X\cap B_R(p)\subset \kgexp(\partial B_{R}(o_p))$. Arguing now as in the previous paragraph, we obtain Theorem~\ref{thm-vol-bound-local}. 

\section{Rigidity}\label{sec:rigidity}

In this section we prove Theorem \ref{main thm}. Assume that $X$ satisfies the conditions of the Theorem. Roughly, there are three key steps (of uneven length and difficulty):
\begin{enumerate}
    \item Prove that the map $\kgexp \rvert_{\partial B_{R}^{\kappa}(o_p)}$ is an isometry between $\partial B_{R}^{\kappa}(o_p)$ and $\partial X$ equipped with the intrinsic metric.
    \item Prove that $X$ must be a convex subset of the model space $M^{m}(\kappa)$.
    \item Analyze the rigidity of convex subsets of the model space $M^{m}(\kappa)$ with radius $R$ and maximal boundary volume.
\end{enumerate}

Let us proceed with the proof. Suppose that $X=\overline{B_R(p)}$ and $\mathcal{H}^{m-1}(\partial X) = \mathcal{H}^{m-1}(\partial B_{R}^{\kappa})$. We start with the following lemma.

\begin{Lem}\label{lem-kgexp-onto}
	The point $p$ is a regular and $\kgexp$ maps $\partial B_{R}^{\kappa}(o_p)$ exactly onto $\partial X$. 
\end{Lem}
\begin{proof}
By volume rigidity for Alexandrov spaces of $curv\ge 1$ it holds that $\mathcal H^{m-1}(\Sigma_p)\le  \mathcal H^{m-1}(\SS^{m-1})$ and the equality  is only possible if $\Sigma_p\cong\SS^{m-1}$.  Hence the by Theorem ~\ref{thm-tancone-vol-bound} the point $p$ must be regular.

	Next, suppose there exists some $v \in \partial B_{R}^{\kappa}(o_p)$ so that $\kgexp(v) \notin \partial X$. Since $\kgexp$ is $1$-Lipschitz and $\partial X$ is closed, we can find some small neighborhood $U$ of $v$ in $\partial B_{R}^{\kappa}(o_p)$ so that $\kgexp(U) \cap \partial X = \varnothing$. Then we have that
\begin{equation*}
	\mathcal{H}^{m-1}(\partial B_{R}^{\kappa}(o_p)) > \mathcal{H}^{m-1}(\partial B_{R}^{\kappa}(o_p) \setminus U) \geq \mathcal{H}^{m-1}(\kgexp(\partial B_{R}^{\kappa}(o_p) \setminus U)) \geq \mathcal{H}^{m-1}(\partial X),
\end{equation*}
by Lemma~\ref{lem-exp-onto-bry}. Since $\mathcal{H}^{m-1}(\partial B_{R}^{\kappa}(o_p))$ is at most $\mathcal{H}^{m-1}(\partial B_R^\kappa)$,  this contradicts assumption \eqref{max bd vol}. 
\end{proof}

Since $\kgexp \rvert_{B_{R}^{\kappa}(o_p)}:B_{R}^{\kappa}(o_p) \to X$ is $1$-Lipschitz by Proposition \ref{exp short}, the map $\kgexp \rvert_{\partial B_{R}^{\kappa}(o_p)}:\partial B_{R}^{\kappa}(o_p) \to \partial X$ must also be $1$-Lipschitz, where $\partial B_{R}^{\kappa}(o_p)$ and $\partial X$ are both equipped with the intrinsic metrics. 

Let $\mathcal{H}_{\partial X}^{m-1}$ and $\mathcal{H}_X^{m-1}$ be the $(m-1)$-Hausdorff measures on $\partial X$ with respect to the intrinsic metric $d_{\partial X}$ and the extrinsic metric $d_X$ respectively. We will often implicitly make use of the following result due to Fujioka who proved the same more generally for extremal subsets \cite[Corollary 3.17]{Fuj-extr}.
\begin{Lem}\label{lem-extr=intr}
The intrinsic and extrinsic $(m-1)$-Hausdorff measures on $\partial X$ are equal:
\[
\mathcal{H}_{\partial X}^{m-1}=\mathcal{H}_X^{m-1}.
\]
\end{Lem}
 Henceforth we will not distinguish between intrinsic and extrinsic  $\mathcal H^{m-1}$ on $\partial X$.
 
Arguing as in Lemma \ref{lem-kgexp-onto} and using \eqref{max bd vol}, Lemma~\ref{lem-extr=intr} implies that $\kgexp \rvert_{\partial B_{R}^{\kappa}(o_p)}: (\partial B_{R}^{\kappa}(o_p), \dd_{\partial B_{R}^{\kappa}(o_p)}) \to (\partial X, \dd_{\partial_X}) $ is $\mathcal{H}^{m-1}$-preserving (in the sense that the $\mathcal{H}^{m-1}$ measure of a set is equal to that of its image). As noted before, it is also $1$-Lipschitz and onto. We will prove from these three facts that it must be an isometry. This type of result is known as Lipschitz-volume rigidity. 

\begin{Rem}\label{alternate proof}
In the case where the domain and target are Alexandrov spaces without boundary such a Lipschitz-volume rigidity statement was proved in \cite{L15}. However, that  result cannot be applied directly since we do not know that $\partial X$ is Alexandrov. Our proof will be simpler than that of \cite{L15} since we know that the domain is actually a smooth manifold without boundary.  In this proof we will use the following consequence of the rigidity result of \cite{GP-almost-maximal}: Let $\kappa = 1$ and $R = \pi/2$.  If $\curv X^m\ge \kappa$ and $\mathcal H^{m-1}(\partial X)= \mathcal{H}^{m-1}(\SS^{m-1})$ then $\partial X$ is isometric to $\mathcal \SS^{m-1}$ with respect to the intrinsic metrics.  

Alternatively, this follows by induction on the dimension of $X$.   More precisely.  Proposition \ref{int isom} in dimension $m-1$ implies Lemma \ref{lem-all-reg-bry} in dimension $m$ which implies Lemma \ref{lem-bry-homeo} in dimension $m$ which implies  Proposition \ref{int isom} in dimension $m$.  Thus the use of  \cite{GP-almost-maximal} for proving Proposition~\ref{int isom} can actually be avoided. 
%However, we believe that by using similar arguments as in \cite{L15},the result of \cite{GP-almost-maximal} is not needed. 
\end{Rem}

We first prove that every point $p \in \partial X$ is regular with respect to the intrinsic metric $\dd_{\partial X}$.
\begin{Lem}\label{lem-all-reg-bry}
	For every point $q \in \partial X$, $T_q \partial X = (\R^{m-1}, \dd_{\R^{m-1}})$. 
\end{Lem}

\begin{proof}
	Fix some $q \in \partial X$ and choose $q' \in \partial B_R^{\kappa}(o_p)$ so that $\kgexp(q') = q$. Since $\kgexp$ is $1$-Lipschitz, we have that $\kgexp(B_R^{\kappa}(o_p) \cap B_r(q')) \subset B_r(q)$ (here we are considering extrinsic metrics). Now using that the map is $\mathcal{H}^{m-1}$-preserving from $\partial B_{R}^{\kappa}(o_p)$ to $\partial X$ (note that it is not $\mathcal{H}^{m-1}$-preserving on all of $B_R^{\kappa}(o_p)$), we have that for arbitrary $r > 0$,
\begin{equation*}
	\mathcal{H}^{m-1}(\partial X \cap B_{r}(q)) \geq \mathcal{H}^{m-1}(\kgexp(\partial B_{R}^{\kappa}(o_p) \cap B_{r}(q'))) = \mathcal{H}^{m-1}(\partial B_{R}^{\kappa}(o_p) \cap B_{r}(q')). 
\end{equation*}
Using that $q'$ is the boundary point of a smooth Riemannian manifold,  as $r \to 0$ we have
\begin{equation*}
	\frac{1}{r^{m-1}}\mathcal{H}^{m-1}(\partial X \cap B_{r}(q)) \geq \frac{1}{r^{m-1}} \mathcal{H}^{m-1}(\partial B_{R}^{\kappa}(o_p) \cap B_{r}(q')) \to \omega_{m-1},
\end{equation*}
where $\omega_{m-1}$ is the volume of the ball of radius $1$ in $\R^{m-1}$. Taking a limit to the tangent cone at $q$ and use the stability of volume for the boundary, for example from \cite{BNS22}[Theorem 1.8] or \cite{Fuj-extr}[Theorem 1.3], we have that 
\begin{equation*}
	\mathcal{H}^{m-1}(B_1(o_q) \cap \partial(T_q X)) \geq \omega_{m-1}.
\end{equation*}	
This means that the space of directions $\Sigma_q$ at $q$, which has $\curv \geq 1$ and is with boundary, has maximal boundary volume. By the rigidity results of \cite{petrun-semiconcave} and \cite{GrPe-sphere-thm}, we conclude that the inequality in the previous equation is an equality and that $\Sigma_q$ must be either a $(m-1)$-dimensional hemisphere or the nontrivial intersection of two such. In both cases, it is easy to see that $T_q \partial X = \partial (C(\Sigma_q))$ must be $\R^{m-1}$ under the intrinsic metric. 
\end{proof}

Next, we show that $\kgexp \rvert_{\partial B_{R}^{\kappa}(o_p)}: \partial B_{R}^{\kappa}(o_p) \to \partial X$, which for notational simplicity we will also denote by $\kgexp$, is a homeomorphism.

\begin{Lem}\label{lem-bry-homeo}
	The map $\kgexp: \partial B_R^{\kappa}(o_p) \to \partial X$ is a homeomorphism. 
\end{Lem}

\begin{proof}
	Since $\kgexp$ is continuous and onto, and $\partial B_R^{\kappa}(o_p)$ is compact, it suffices to show that $\kgexp$ is one-to-one. Fix $q \in \partial X$ and suppose for the sake of contradiction that there exist different points $q_1, q_2 \in \partial B_R^{\kappa}(o_p)$ so that $\kgexp(q_1) = \kgexp(q_2) = q$. Using that that $\kgexp$ is $\mathcal{H}^{m-1}$-preserving, we have that
\begin{align*}
	\frac{1}{r^{m-1}}\mathcal{H}^{m-1}(B_r(q)) &\geq \frac{1}{r^{m-1}} \mathcal{H}^{m-1}(\kgexp(B_r(q_1) \cup B_r(q_2)))\\ &= \frac{1}{r^{m-1}} \mathcal{H}^{m-1}(B_r(q_1) \cup B_r(q_2)) \to 2\omega_{m-1},
\end{align*}
for sufficiently small $r \to 0$, where the balls considered are with intrinsic distance. This is a contradiction since the intrinsic ball of radius $r$ around $q$ in $\partial X$ is contained in the intersection of $\partial X$ and the extrinsic ball of radius $r$ around $q$, and so the lefthand side can converge to no more than $\omega_{m-1}$ by the previous lemma. 
\end{proof}

Finally, we will prove that the map is an isometry. For clarity, in the following we will use $B_r(\cdot)$ to indicate extrinsic balls and $B^{\partial}_r(\cdot)$ to indicate intrinsic balls.

\begin{Pro}\label{int isom}
	The map $\kgexp: \partial B_R^{\kappa}(o_p) \to \partial X$ is an isometry with respect to the intrinsic metrics.
\end{Pro}

\begin{proof}
	Fix $q \in \partial X$ and its preimage $q' \in  \partial B_R^{\kappa}(o_p)$ under $\kgexp$. Obviously, we have that as $r \to 0$,
\begin{equation}\label{bry-reg-hm-1}	 
\frac{1}{r^{m-1}}\mathcal{H}^{m-1}(B_r^\partial (q')) \to \omega_{m-1}.
\end{equation}	 
We claim the same holds for $q$, that is, as $r \to 0$,
\begin{equation}\label{bry-reg-hm-2}
\frac{1}{r^{m-1}}\mathcal{H}^{m-1}(B_r^\partial (q)) \to \omega_{m-1}.
\end{equation}
Indeed, we have that $(X, \partial X, \frac{1}{r}\dd_X, q)\to (T_qX, T_q \partial X, \dd_{T_qX}, o_q)$ as $r\to 0$, so in particular by \cite[Theorem 1.8]{BNS22} or \cite[Theorem 1.3]{Fuj-extr}, $\mathcal{H}^{m-1}$ on the boundaries converge as well. By Theorem~\ref{intr-bdry-conv}, the boundaries converge under intrinsic metrics, i.e., $(\partial X, \frac{1}{r}\dd_{\partial X}, q)\to (T_q\partial X, \dd_{T_q \partial X}, o_q)$ as $r \to 0$. Combining this with the fact that $T_q \partial X = \R^{m-1}$ (Lemma \ref{lem-all-reg-bry}) and the invariance of $\mathcal{H}^{m-1}$ with respect to intrinsic and extrinsic metric (Lemma \ref{lem-extr=intr}) gives that
\begin{equation*}
\frac{1}{r^{m-1}}\mathcal{H}^{m-1}(B_r(q) \cap \partial X) \to \omega_{m-1}.
\end{equation*}
To prove \eqref{bry-reg-hm-2} from this all we need is the following lemma which says that for any $\delta > 0$, the intrinsic ball of radius $r$ contains the extrinsic ball of radius $r/(1+\delta)$ for sufficiently small $r$. 
\begin{lem}\label{lem-extr-intr-extr}
Let $(X^n, \dd)$ be an Alexandrov space, $E\subset X$ be an extremal subset and $p\in E$ be a point. Let $\dd_E$ be the intrinsic metric on $E$ induced by the extrinsic metric $d$.
Then 
\[
\lim_{q\to p, q\in E}\frac{\dd_E(p,q)}{\dd(p,q)}=1
\]
\end{lem}
\begin{proof}[Proof of Lemma \ref{lem-extr-intr-extr}]
Let $f=\dd(\cdot, p)$. It is well known that $\lim_{x\to p} |\nabla_x f|=1$. Let $\delta>0$. By above there is $r>0$ such that $ |\nabla_x f|>1-\delta$ for any $x\in B_{r}(p)\setminus \{p\}$.

Let $q\in E \cap B_{r}(p)\setminus \{p\}$. By \cite[Lemma 6.1]{Fuj-uniform-bounds-extr} there exists a backward gradient curve $\gamma$ of $f$ passing through $q$ which lies in $E$. The function $f$ decreases with speed at least $(1-\delta)^2$ along this curve. This implies that $\gamma$ reaches $p$ in finite time  $\le  \frac{\dd(p,q)}{(1-\delta)^2}$. Since $\gamma$ is 1-Lipschitz this gives that $\dd_E(p,q)\le \frac{\dd(p,q)}{(1-\delta)^2}$.

\end{proof}
	
	Now for each $r > 0$ consider the $1$-Lipschitz pointed maps $\Phi_r: (\partial B^{\kappa}_R(o_p), \frac{1}{r}\dd_{\partial B^{\kappa}_R(o_p)}, q') \to (\partial X, \frac{1}{r}\dd_{\partial X}, q)$ which is just a rescaling of $\kgexp$. Observe that the unit ball around $q'$ (with respect to $\frac{1}{r}\dd_{\partial B^{\kappa}_R(o_p)}$) gets more and more dense in the unit ball around $q$ (with respect to $\frac{1}{r}\dd_{\partial X}$). This follows by volume considerations from \eqref{bry-reg-hm-1}, \eqref{bry-reg-hm-2} and the observation that for any $s > 0$ there is a uniform lower bound on the intrinsic ball of radius $s$ around any $x \in \partial X$ (since it contains the image of the intrinsic ball of radius $s$ around $(\kgexp)^{-1}(x)$ and $\kgexp$ is $\mathcal{H}^{m-1}$-preserving).
	
	 Taking now a subconvergent limit of $\Phi_r$, we obtain some $1$-Lipschitz $\Phi: T_{q'} \partial B_R(o_p) \to T_q \partial X$. Moreover, for any $x \in T_{q'}$ and $r > 0$, $\Phi$ maps $B_r(x)$ onto $B_r(\Phi(x))$ by the result of the previous paragraph. Since both domain and codomain are isometric to $\R^{m-1}$, this implies that $\Phi$ is an isometry. 
	 
	 Fix $\delta > 0$. We claim that $(\kgexp)^{-1}$ does not increase the length of curves by more than a factor of $(1+\delta)$. This will obviously imply that $(\kgexp)^{-1}$ is $1$-Lipschitz and hence an isometry. 
	 
	 By a standard partitioning argument, it suffices to prove the following claim: For all sufficiently small $r$ (depending on $q$ and $\delta$), if $\dd_{\partial X}(x, q) = r$ then $\dd_{\partial B^{\kappa}_R(o_p)}((\kgexp)^{-1}(y), q')\leq (1+\delta)r.$
	 
	 Since the map $\kgexp: \partial B_R^{\kappa}(o_p) \to \partial X$ is a homeomorphism the image of the metric sphere (with respect to intrinsic metric), we have that $\partial B^\partial_{(1+\delta)r}(q')$ separates $\partial X$. 
	 
	 Consider any shortest $[qy]$ in $\partial X$ with respect to $\dd_{\partial X}$ of length $r$. Since the map $\Phi$ above is an isometry we get that if  $r$ is sufficiently small  this curve cannot intersect $\kgexp(\partial B^\partial_{(1+\delta)r}(q'))$. Since the latter separates 
$\partial X$ it follows that the entire shortest $[q, y]$ including the point $y$ lies in $\kgexp (B_{(1+\delta)r}^\partial (q'))$. This proves the claim and hence finishes the proof of Proposition \ref{int isom}.
\end{proof}

	Now we can prove that $X$ must be a convex subset of the model space $M^{m}(\kappa)$.
\begin{Pro}\label{pro-conv-subset}
	$X$ is isometric to  a convex subset of the model space $M^{m}(\kappa)$.
\end{Pro}

\begin{proof}
	Fix any $x, y \in X$. Let $d = \dd(x,y)$ and fix a unit speed geodesic $\gamma:[0, d] \to X$ between $x$ and $y$. For each $t \in [0,d]$, we choose a direction $v(t) \in \Sigma_p = \mathbb{S}^{m-1}$ so that $\gamma(t)$ may be reached from $p$ by a geodesic with direction $v(t)$ and length $s(t)$. In general, $v(t)$ may not be unique or continuously dependent on $t$, but $s(t)$ is obviously $1$-Lipschitz. 
	However, in our case $v(t)$ must be unique. This  immediately follows from Lemma~\ref{lem-bry-homeo}  which implies that the map $\kgexp: \partial B_R^{\kappa}(o_p) \to \partial X$ is injective.
	
	Since $\Sigma_p = \mathbb{S}^{m-1}$, we may identify $T_p$ isometrically with the tangent space at some designated origin $p' \in M^{m}(\kappa)$. Taking the Riemannian exponentiation map at $p'$, we define $x' = \text{exp}_{p'}(v(0), s(0))$ and $y'=\text{exp}_{p'}(v(d), s(d))$. Denoting $d' = \dd_{M^m(\kappa)}(x', y')$, we will prove that $d=d'$ and that $\tilde\gamma(t):=\text{exp}_{p'} (v(t), s(t))$ is necessarily the geodesic between $x'$ and $y'$. We will prove this in the case that $s(t) \neq 0$ for any $t \in [0,d]$. As can be seen, the other case will follow easily from the rest of our argument by dividing the geodesic into two parts at $p$ and using a limiting argument.
	
	We first show that $v(t)$ is Lipschitz on $[0,d]$. First consider the curve $\gamma_R(t):=\kgexp((v(t), R))$ in $\partial X$, where we identify $(v(t), R)$ to a point in $T_p = C(\Sigma_p)$ in the obvious way. It is easy to see that this curve must be Lipschitz by using Proposition \ref{exp comp} to compare $\dd_{X}(\kgexp((v(t_1), R),$ $\kgexp((v(t_2), R))$ and $\dd_X(\kgexp((v(t_1), s(t_1)), \kgexp((v(t_2), s(t_2))) = \dd(\gamma(t_1), \gamma(t_2)) = |t_2 - t_1|$, giving the bound 
\begin{align*}
	\dd_{X}(\kgexp((v(t_1), R), \kgexp((v(t_2), R)) \leq C_{\kappa, s_0, R} |t_2 - t_1|,
\end{align*}
where $s_0 = \min_{t\in[0,d]} s(t) > 0$ by assumption. Now using that $\kgexp$ is an intrinsic isometry between $\partial B_R^{\kappa}(o_p)$ to $\partial X$ by Proposition \ref{int isom}, we have that $\tilde \gamma_R(t):=(v(t), R)$ seen as a curve in $\partial B_R^{\kappa}(o_p)$ must also be Lipschitz on $[0,d]$. This immediately implies that $v(t)$ is Lipschitz viewed as a curve in $\Sigma_p$.

We now show that $d' = d$ and that $\text{exp}_{p'} (v(t), s(t))$ is the geodesic between $x'$ and $y'$.
WLOG we can assume that $x\notin [py]$, $y\notin [px]$ and $p\notin [xy]$  since in any  of these cases the equality $d=d'$ is obvious.

	 Using Proposition \ref{int isom} again, we see that the curves $\gamma_R$ and $\tilde\gamma_R$ must be of the same length. 
	 
	 Denoting by $L(\cdot)$ the length of a locally Lipschitz curve, we claim the following.
	 
	 {\bf Claim } % Using Proposition \ref{exp comp}, we see that this implies 
 $L(\tilde \gamma)\le L(\gamma)=d$.
	
	Let us prove the Claim.  Recall that a metric speed of a locally Lipschitz curve $\eta(t)$ in a  metric space $Y$ is defined as $|\eta'(t)|=\lim_{\eps\to 0}\frac{\dd(\eta(t), \eta(t+\eps))}{|\eps|}$. Note that this notation does not mean that any kind of tangent vector to $\eta$ exists. We have that the curves $\gamma, \tilde\gamma, \gamma_R$ and $\tilde\gamma_R$ are all Lipschitz. Therefore they have metric speeds defined for almost all $t$ and their lengths are given by the integrals of the metric speeds \cite[Theorem 2.7.6]{BBI-book} (this is obvious for $\gamma$ because it is a geodesic). We denote by $I$ the full measure the full measure set of $t \in (0,d)$ for which all four curves $\gamma, \tilde \gamma, \gamma_R, 
\tilde \gamma_R$ have metric derivatives.

Next we will show that $|\tilde \gamma'(t)|\le |\gamma'(t)|$ for all $t \in I$ which will obviously imply the Claim.
Let us fix $t\in I$ and let $a=|\tilde \gamma_R'(t)|$.

By above we have that $\dd(\tilde \gamma_R(t+\eps), \tilde \gamma_R(t))=a|\eps|\pm o(|\eps|)$ and the same holds for $\gamma_R$.
Therefore by the cosine law in $M^{m}(\kappa)$ we have that comparison angles $ \kang( \dd(\gamma_R(t), \gamma_R(t+\eps)); R, R) $ and $ \kang( \dd(\tilde \gamma_R(t), \tilde \gamma_R(t+\eps)); R, R) $ differ by at most $o(|\eps|)$. 

Note that $ \kang(\dd(\tilde \gamma_R(t), \tilde \gamma_R(t+\eps)); R, R) = \kang( \dd(\tilde \gamma(t), \tilde \gamma (t+\eps)); s(t), s(t+\eps))$.
On the other hand, by Proposition ~\ref{exp comp} it holds that 

\[
\kang (\dd(\gamma(t), \gamma(t+\eps)); s(t), s(t+\eps)) \ge  \kang(\dd(\gamma_R(t), \gamma_R(t+\eps)); R, R).
\]
Combining the above we get 
\[
\kang (\dd(\gamma(t), \gamma(t+\eps)); s(t), s(t+\eps))\ge  \kang(\dd(\tilde \gamma(t), \tilde \gamma (t+\eps)); s(t), s(t+\eps))-o(|\eps|)
\]	
and hence

\begin{equation}\label{eq:est1}
\dd(\gamma(t), \gamma(t+\eps))\ge \dd(\tilde\gamma(t), \tilde\gamma(t+\eps))-o(|\eps|).
\end{equation}
By the definition of metric speed $\dd(\gamma(t), \gamma(t+\eps))=|\gamma'(t)|\eps+o(|\eps|)$ and  $\dd(\tilde \gamma(t), \tilde \gamma(t+\eps))=|\tilde\gamma'(t)|\eps+o(|\eps|)$. Plugging these into \eqref{eq:est1}, dividing by $|\eps|$ and taking the limit as $\eps\to 0$ gives
\[
|\gamma'(t)|\ge |\tilde\gamma'(t)|.
\]
Integrating from $0$ to $d$ gives the Claim.

	Now $\tilde \gamma$ is a curve between $x'$ and $y'$ and so it has length at least $d'$, which means $d' \leq d$. However, standard comparison tells us that $d' \geq d$. This means that $d = d'$ and $\tilde\gamma$ must be the geodesic from $x'$ and $y'$ as required. This means that the map $x\mapsto \log_p([px])$ gives a distance preserving embedding of $X$ into the $\pik/2$ ball around $o_p$ in $T_pX$ equipped with the constant curvature $\kappa$ metric and moreover the image is convex. Thus $X$ is isometric to  a convex subset of a $\pik/2$ ball in $M^{m}(\kappa)$.
\end{proof}

To finish the proof we need the following Theorem.
\begin{thm}\label{nested-convex}
Let $R\le \pik/2$ and let $p\in M^{m}(\kappa)$. Let $K_1\subset K_2\subset \overline{B_R(p)}$ be nested closed convex domains. Then
\begin{equation}\label{nonstrict-ineq}
\mathcal H^{m-1}(\partial K_1)\le \mathcal H^{m-1}(\partial K_2)
\end{equation}

Moreover, this inequality is strict if $R<\pik/2$ and $K_1\ne K_2$, i.e.,
\begin{equation}\label{strict-ineq}
\mathcal H^{m-1}(\partial K_1)< \mathcal H^{m-1}(\partial K_2)
\end{equation}
\end{thm}	
\begin{rem}
This theorem is likely known as it easily follows from known results but we could not find a reference for $\kappa>0$.  Therefore we give a proof.
\end{rem}
\begin{rem}
The strict inequality can fail if  $\kappa>0$ and $R=\pik/2$ even if $K_1\ne K_2$. An example is when $K_1$ is an intersection of two hemispheres in $\SS^n$ and $K_2$ is one of these hemispheres.
\end{rem}
\begin{proof}

Theorem \ref{nested-convex} easily follows from the following version of Crofton's formula in high dimensions \cite[14.70]{santalo-kac-book}:
Given a closed convex domain $K\subset B_{\pik/2}(p)$ in $M^m(\kappa)$ with piecewise smooth boundary, the volume of the boundary of $K$ can be computed by the formula
\begin{equation}
\mathcal H^{m-1}(\partial K)=C(m)\cdot\int_{G(m)}\#\{ L\cap \partial K\}dL
\end{equation}
where $G(m)$ is the set of lines in  $M^m(\kappa)$ (or great circles if $\kappa>0$), $L\in G(m)$  and $dL$ is some naturally defined volume measure on $G(m)$.

Since the convex domains $K_1\subset K_2$ are nested it follows that $\#\{ L\cap \partial K_1\}\le \#\{ L\cap \partial K_2\}$ which immediately gives the non-strict inequality in Theorem \ref{nested-convex} assuming that $K_1$ and $K_2$ have piecewise smooth boundaries.

 When $R<\pik/2$ the intersections $L\cap  K_1$ and $L\cap  K_2$ are intervals if nonempty and it is not hard to see that there is an open set of lines which intersect $K_2$ but not $K_1$ which gives the strict inequality \eqref{strict-ineq}.

The case of general $K_1, K_2$ easily follows by approximation.
\end{proof}
\begin{rem}
It is instructive to think about why the above argument doesn't give a strict inequality if $\kappa>0$ and $R=\pik/2$ even if $K_1\ne K_2$. This should give an alternate proof  to \cite{GP-almost-maximal} of  the rigidity classification in this case. 
\end{rem}

\begin{rem}
A different proof of Theorem~\ref{nested-convex} for $\kappa\le 0$ follows from a classical result that in a $CAT(0)$ space  the nearest point projection onto a closed convex subset is $1$-Lipschitz.
For $\kappa>0$ it can be shown using the proof of Theorem 1.1 in  \cite{Lyt-Pet-short-retr} which implies that  there exists a strong deformation retraction $\Pi_t\co K_2\to K_2, t\in [0,1]$ where $\Pi_0=\Id, \Pi_1\co K_2\to K_1 $ and $\Pi_t|_{K_1}=\Id$ for all $t$.  Moreover, $\Pi_t$ is 1-Lipschitz for all $t$. Moreover if $R<\pik/2$ then the constructed map $\Pi_1$ has local Lipschitz constant strictly smaller than 1 outside $K_1$. This gives the strict bound \eqref{strict-ineq}.

\end{rem}

We now complete the proof of the rigidity theorem. 

\begin{proof}[Proof of the first part of Theorem~\ref{main thm}]
By Proposition \ref{pro-conv-subset}, $X$ is a convex subset in $M^{m}(\kappa)$. It is clearly contained in the ball $\overline{B_R^{\kappa}(p)}$ in $M^{m}(\kappa)$.
If it is not equal to the whole ball and $R<\pik/2$ then by Theorem~\ref{nested-convex} $\mathcal H^{m-1}(\partial X)<\mathcal{H}^{m-1}(\partial B_{R}^{\kappa})$, which is a contradiction. This finishes the proof of Theorem \ref{main thm} in that case. 

If $\kappa>0$ and $R=\pik/2$ then the result follows from \cite{GP-almost-maximal}.
\end{proof}

The second part of Theorem~\ref{main thm}, i.e., when $p \in \partial X$, $\overline{B_R(p)}=X$, and $\mathcal{H}^{m-1}(\partial X)=\mathcal{H}^{m-1}(\partial B^{\kappa}_{R,+})$, can also be proved in a similar manner with some minor changes. We will not go into full detail of the proof but instead give an outline indicating the necessary changes to the previous proof.

Suppose that the conditions of the theorem are met. 

First, one can show that $T_pX = \R^{m}_+$ and that the map $\gexp\rvert_{\partial B_R^{\kappa}(o_p)}$ maps exactly onto the $\partial X$. Indeed, arguing as in the proof of Lemma \ref{lem-kgexp-onto}, Lemma \ref{lem-exp-onto-bry} implies that $\mathcal{H}^{m-1}(\partial B^{\kappa}(o_p)) = \mathcal{H}^{m-1}(\partial B^{\kappa}_{R,+})$. It follows that the double (see Theorem \ref{thm-doubling}) of $\Sigma_p X$ must have the same measure as $\mathbb{S}^{m-1}$ and hence by the maximal volume rigidity for Alexandrov spaces with $\curv \ge 1$ must be isometric to $\mathbb{S}^{m-1}$. Therefore, $\Sigma_p X$ is the hemisphere and $T_p X = \R^{m}_+$. The proof that $\gexp\rvert_{\partial B_R^{\kappa}(o_p)}$ maps exactly onto $\partial X$ is exactly the same. 

Next, one can show that that map $\kgexp: \partial B_R^{\kappa}(o_p) \to \partial X$ is an isometry. The proofs of Lemma \ref{lem-bry-homeo} and \ref{int isom} goes through without change. The only place where one has to be careful is that $B_R^{\kappa}(o_p)$ is not a manifold with smooth boundary anymore. However, it is clear that for any singular $q$, i.e., if $q$ is exactly on the boundary of the $R$-hemisphere of $o_p$, one still has that $q$ is a regular point of $\partial B_R^{\kappa}(o_p)$ and hence  $\mathcal{H}^{m-1}(B_r(q) \cap \partial X) \to \omega_{m-1}$ as $r \to 0$, which is all that is needed. 

Finally, the same proof for Proposition \ref{pro-conv-subset} shows that $X$ must be a convex subset of $B_{R, +}^{\kappa} \subset M^{m}(\kappa)$. Theorem \ref{nested-convex} now gives the desired result in the cases where $\kappa \leq 0$ or $\kappa > 0$ and $R < \pik/2$. In the case where $\kappa > 0$ and $R = \pik/2$, one has $\mathcal{H}^{m-1}(\partial B_{R, +}^{\kappa}) = \mathcal{H}^{m-1}(\partial B_{R}^{\kappa})$, so one is actually in the original rigidity situation addressed in \cite{GP-almost-maximal}. It is then easy to see that if $p \in \partial X$ and $X = \overline{B_{\pik/2}(p)}$ then $X$ must be $L^{m, \kappa}_{\alpha}$ for some $\alpha \in (0, \pi/2]$. 

\section{Almost Rigidity}

Given the rigidity result we obtain in Theorem \ref{main thm} it is natural to wonder what happens in the almost rigidity situation. We will consider the first part of Theorem \ref{thm-almost-rigidity} first. We will divide our analysis into two cases, when $R < \pik/2$ and when $R = \pik/2$. 

We consider the case $R < \pik/2$ first and prove the following.
\begin{Thm}\label{thm-almost-rigid}
Given an integer $m$, $\kappa \in \R$ and $R > 0$. If $\kappa > 0$, assume in addition that $R < \pik/2$. For any $\eps > 0$, there exists $\delta(\kappa, m, R) > 0$ such that the following holds: Let $X^m$ be Alexandrov with $\curv \ge \kappa$. Assume $\overline{B_R(p)} = X$ for some $p \in X$ and that 
\begin{equation*}
	\mathcal{H}^{m-1}(\partial X) > \mathcal{H}^{m-1}(\partial B_{R}^{\kappa})-\delta.
\end{equation*}
Then
\begin{equation*}
	\dd_{GH}(X, \overline{B_R^{\kappa}}) < \eps.
\end{equation*}
\end{Thm}

\begin{Rem}
Since  $B_R^{\kappa}$ is a topological disk for any $\kappa \in R$ and $R<\pik/2$, by Perelman's Stability Theorem \cite{Per1, Ka07} it follows that in the setting of the above theorem $X$ must be a topological disk too provided $\delta$ is sufficiently small.
\end{Rem}

\begin{proof}
Suppose the theorem is false. Then there exist $\epsilon>0$ and a contradicting sequence $(X^m_i, \dd_i)$ of Alexandrov spaces with $\curv\ge\kappa$ and $p_i\in X_i$ such that $B_R(p_i)=X_i$,  $\mathcal{H}^{m-1}(\partial X_i) > \mathcal{H}^{m-1}(\partial B_{R}^{\kappa})-\delta_i$ where $\delta_i \to 0$ as $i\to \infty$ but $\dd_{GH}(X_i, \overline{B_R^\kappa})\ge \eps$ for all $i$.

By compactness, passing to a subsequence we may assume that $(X_i, \dd_i, p_i)\to (X_\infty, \dd_\infty, p_\infty)$ in Gromov-Hausdorff topology for some $X_\infty$. We claim the following:

{\bf Claim} $\dim X_\infty=m$, i.e., this sequence is noncollapsing. 

Assuming the claim is true we see that $X_\infty$ also has $\curv\ge \kappa$ and satisfies $B_R(p_\infty)=X_\infty$. From the stability of $\mathcal H^{m-1}$ from \cite{BNS22}[Theorem 1.8] or \cite{Fuj-extr}[Theorem 1.1] we have that $\mathcal H^{m-1}(\partial X_i)\to \mathcal H^{m-1}(\partial X_\infty)$ and hence by Theorem~\ref{thm-vol-bound} we have that $\mathcal H^{m-1}(\partial X_\infty)=\mathcal{H}^{m-1}(\partial B_{R}^{\kappa})$. Now by Theorem \ref{main thm}, it follows that $X$ is isometric to $\overline{B_{R}^{\kappa}}$ . This is a contradiction with the assumption  $\dd_{GH}(X_i, \overline{B_R^\kappa})\ge \eps$ for all $i$.

Thus to finish the proof of the theorem we need only to establish the Claim above that the sequence does not collapse.

Suppose that it is false and $\dim X_\infty=k<m$. 

For each $i$ let $d_i=\dd_i(p_i, \partial X_i)$.
There are two possibilities:

{\bf Case 1} $\liminf_{i\to\infty}d_i=0$
or

{\bf Case 2} $d_i\ge d>0$ for all $i$ for some $d>0$.

Let us treat {\bf Case 1} first.

By passing to a subsequence we can assume $d_i\to 0$. Then there exists $q_i\in \partial X_i$ such that $d(p_i,q_i)\le d_i$. Then $B_{R+d_i}(q_i) = X_i$, and so by Theorem~\ref{thm-vol-bound}  we have that $ \mathcal{H}^{m-1}(\partial X_i) \le \mathcal{H}^{m-1}(\partial B_{R+d_i,+}^{\kappa} )$.

But $\mathcal{H}^{m-1}(\partial B_{R+d_i,+}^{\kappa} )\to \mathcal{H}^{m-1}(\partial B_{R,+}^{\kappa} )<\mathcal{H}^{m-1}(\partial B_{R}^{\kappa} )$.
This contradicts the assumption that $\mathcal{H}^{m-1}(\partial X_i) > \mathcal{H}^{m-1}(\partial B_{R}^{\kappa})-\delta_i$ where $\delta_i \to 0$. Therefore, Case 1 cannot occur.

We remark here that this argument fails if $\kappa>0, R=\pik/2$ since in this case $\mathcal{H}^{m-1}(\partial B_{R,+}^{\kappa} )=\mathcal{H}^{m-1}(\partial B_{R}^{\kappa} )$.

We now consider {\bf Case 2}. Clearly we will have shown that Case 2 is impossible as soon as we prove the following.

{\bf Claim} $\partial X_i\to X_\infty$ with respect to the ambient metrics on $\partial X_i$. 

To prove the claim we will need the following lemma. 

\begin{lem}\label{extr-dim-proper}
Let $Y^n$ be an Alexandrov space and $E\subset Y$ be extremal. If $\dim E=n$ then $E=Y$.
\end{lem}
Here by $\dim E$ we mean either the topological or the Hausdorff dimension of $E$. They are the same   by \cite{Pet-Per-extremal}. 
\begin{proof}
If $\dim E=n$ then by \cite[Theorem 1.1]{Fuj-extr} there is a regular point $p\in Y$ such that a small ball around $p$ is contained in $E$. Let $y\in Y$ be arbitrary. Consider a shortest geodesic $[py]$. Then there is a point $x$ on this geodesic close to $p$ which is in $E$. let $f=d^2(\cdot, p)$. It is semiconcave and since $E$ is extremal it is invariant under the gradient flow $\phi_t$ of $f$. Hence $\phi_t(x)\in E$ for any $t\ge 0$. Since $x\in [py]$ there exists $t\ge 0$ such that $\phi_t(x)=y$. hence $y\in E$.
\end{proof}

Let us proceed with the proof of the {\bf Claim} above.

 After passing to a subsequence $\partial X_i$ subconverges to an extremal subset $E$ of $X_\infty$, see, for example, \cite{petrun-semiconcave}[Lemma 4.1.3]. If this is a proper subset then by Lemma~\ref{extr-dim-proper} $\dim E<\dim X_\infty<m$ which means that $\dim E\le m-2$.
 % Note that there are several equivalent ways to define $\dim E$. For concreteness, one may take, for example, the Hausdorff dimension. 
 Theorem~\ref{thm-vol-bound-local} means that there is a universal $c=c(m)>0$ such that for all $x\in X_i, r\le \pi$ it holds that $\mathcal H^{m-1}(\partial X_i\cap B_r(x))\le c(m)r^{m-1}$. Now a simple covering argument implies that $\mathcal H^{m-1}(\partial X_i)\to 0$.
Indeed, $E$ has finite  $\mathcal H^{m-2}$ by \cite[Theorem 1.1]{Fuj-uniform-bounds-extr}. Hence there is a $C>0$ such that for any $\eps>0$, $E$ can be covered by countably many balls $B_{r_j}(y_j)$ with $r_i<\eps$ such that $\sum_j r_j^{m-2}<C$. Since $E$ is compact this collection of balls can be made finite so that $E\subset \cup_{j=1}^NB_{r_j}(y_j)$. Note that there exists  a small $\mu>0$ such that $ \cup_{j=1}^NB_{r_j}(y_j)$ contains the $\mu$-neighborhood of $E$.

These balls can be lifted to nearby balls $B_{r_j}(y_j^i)$ in $X_i$ which will cover $\partial X_i$ for all large $i$. Therefore, 
\[
\mathcal H^{m-1}(\partial X_i)\le \sum_{j=1}^N \mathcal H^{m-1}(\partial X_i\cap B_{r_j}(y_j^i) )\le \sum_{j=1}^N c(m) r_j^{m-1}\le \sum_{j=1}^N c(m) \eps r_j^{m-2}<\eps c(m) C
\]
for all large $i$. This shows that $\mathcal H^{m-1}(\partial X_i)\to 0$ if $E\ne X_\infty$, which is a contradiction. Therefore, Case 2 is impossible.

This proves the Claim that the convergence $X_i\to X$ is noncollapsing which finishes the proof of Theorem~\ref{thm-almost-rigid}.
\end{proof}

We now consider almost rigidity in the case $\kappa > 0, R= \pik/2$. We will prove
\begin{Thm}\label{thm-almost-rigid-2}
Given an integer $m$ and some $\kappa > 0$. Let $R = \pik/2$. For any $\eps > 0$, there exists $\delta(\kappa, m) > 0$ such that the following holds: Let $X^m$ be Alexandrov with $\curv \ge \kappa$. Assume $\overline{B_R(p)}=X$ for some $p \in X$ and that
\begin{equation*}
	\mathcal{H}^{m-1}(\partial X) > \mathcal{H}^{m-1}(\partial B_{R}^{\kappa})-\delta.
\end{equation*}
Then there exists $Y$, which is either $L^{m, \kappa}_{\alpha}$ for some $\alpha \in (0, \pi]$ or the closed ball of radius $R$ in $M^{m-1}(\kappa)$, so that
\begin{equation*}
	\dd_{GH}(X, Y) < \eps.
\end{equation*}
\end{Thm}
 
\begin{proof}
To unburden the exposition we skip some technical details in the proof. 

By rescaling we need only consider the case $\kappa = 1$. Suppose we have $\mathcal H^{m-1}( \partial X_i)\to \mathcal H^{m-1}( \SS^{m-1})$. If $\dim X_\infty=m$ then Theorem \ref{main thm} gives that $X_\infty$ is $L^{m,1}_{\alpha}$ for some $\alpha \in (0, \pi]$.

Suppose then that $\dim X_\infty=k<m$. The local bound given by Theorem~\ref{thm-vol-bound-local} means that there is a universal $c=c(m)>0$ such that for all $x\in X_i, r\le \pi$ it holds that $\mathcal H^{m-1}(\partial X_i\cap B_r(x))\le cr^{m-1}$. This  easily implies that if $k<m-1$ then $\mathcal H^{m-1}(\partial X_i)\to 0$ (see for example the computation in the previous proof). This gives a contradiction and hence $k<m-1$ is impossible. Thus $k=m-1$ is the only possibility if collapse does occur.

By the claim above $\partial X_i\to X_\infty$ with respect to ambient metrics, i.e., $\partial X_i$ becomes denser and denser in $X_i$. 

Let $q_\infty\in  X_\infty$ be a regular point. Fix a small $\delta>0$.  Pick $q_i\in  X_i$ converging to $q_\infty$. Then by looking at an $\epsilon$-strainer map $F\co X_\infty\to \R^{m-1}$  near $q_\infty$ and lifting it to $F_i\co X_i\to \R^{m-1}$ we get that for all large $i$ near $q_i$, $F_i$ is a bundle map with fiber a compact connected $1$-dimensional MCS-space. This means it's a finite graph.

By Corollary~\ref{cor-bry-conical} it cannot contain any vertices of degree $\ge 3$ and hence the fiber must be an interval.  The boundary points of the fiber intervals are exactly the points in  $\partial X_i$. Thus, restricted to $\partial X_i$ near $q_i$ we get that $F_i$ is a $2$-fold product cover of a disk  in $\R^{m-1}$ and on each connected component it is a $(1+\delta)$ bilipschitz homeomorphism \cite[Proposition 3.3]{Fuj-extr}.  

 The set $S=X_\infty\setminus X_\infty(m-1,\delta)$, where $X_\infty(m-1,\delta)$ denote the points which are (m-1) $\delta$-strained, is compact and has Hausdorff dimension $\le m-2$. Take a small $\eps>0$. By the same argument as in the proof of Case 2 earlier we can find finitely many open balls $B_{r_j}(y_j), j=1,\ldots ,N$, covering $S$ such $\sum_{j=1}^N r_j^{m-1}<\eps$ and hence  the corresponding balls $B_{r_j}(y_j^i)$ have the property that $\mathcal H^{m-1}(\partial X_i\cap (\cup_j B_{r_j}(y_j^i)))\le c(m) \eps$.
Outside of the union of these balls by a standard argument \cite{BGP, Yam-conv-theorem} the maps $F^{-1}\circ F_i$ can be glued into a global map  $\Phi_i\co X_i \to X_\infty$ which is locally $(1+\delta)$-bilipschitz from $\partial X_i\setminus  \cup_j B_{r_j}(y_j^i)$ to $X_\infty \setminus  \cup_j B_{r_j}(y_j)$. By above it is also 2-to-1. Since $\eps$ and $\delta$ can be chosen to be arbitrarily small this gives that 
 $\mathcal H^{m-1}(X_\infty)=\frac{1}{2} \mathcal H^{m-1}(\SS^{m-1})$.

Since $\curv X_\infty\ge 1$ and $\overline{B_{\pik/2}(p_\infty)}=X_\infty$, by absolute volume comparison we have that $\mathcal H^{m-1} (B(p_\infty, \pi/2))\le \frac{1}{2} \mathcal H^{m-1}(\SS^{m-1})$. Thus we have an equality in the absolute volume comparison. This gives \cite[Theorem 5.2]{L15} that $X_\infty$ must be either isometric to a hemisphere in $\SS^{m-1}$ or to a quotient of  such hemisphere by an isometric involution of its boundary.
This can also be thought of the quotient of $\SS^{m-1}$ by an isometric involution of $\R^{m}$ which is given by a diagonal matrix with $\pm 1$ on the diagonal.

 Such quotients are easily understood depending on the number of $-1$'s on the diagonal. There must be at least one $-1$. If it's exactly one we get a hemisphere $\SS^{m-1}_+$. If all are $-1$'s we get $\RP^{m-1}$.
 In between we get spherical joins of $\RP^k$ and $\SS^l$ where $k+l=m-2$.
 
 A hemisphere can occur as the limit of Grove-Petersen examples.
The case  $X_\infty=\RP^{m-1}$ can be ruled out because it's not simply connected and $X_\infty$ must be simply connected since all $X_i$ are (see Lemma~\ref{lem-conv-pi1} below).  It remains to rule out  the intermediate cases. 

It is well known that  any compact Alexandrov space $Y$  admits a nondecreasing contractibility function $\rho\co \R_+\to \R_+$ such that $\rho(r)\ge r$ and $\rho(r)\to 0$ as $r\to 0$. This means that there is $r_0>0$ such that  for all any $r<r_0$ every ball $B_r(y)$ in $Y$ is contractible inside the ball $B_{\rho(r)}(y)$. Moreover by \cite{petrun-semiconcave}, \cite{Kap-noncol} or \cite{Yam-lipsch-contr} there is a stronger contractibility statement: every ball $B_r(y)$ is contained in a closed set $C_r(y)$ of diameter $\le \rho(r)$ where $C_r(y)$  is a superlevel set of a strongly concave function. In particular all sets $C_r(y)$  and their intersections are strongly convex and contractible. In particular $X$ is homotopy equivalent to the nerve of any finite covering by  $C_r(y)$'s. For a stronger statement see \cite{Yam-good-cov}.

Also, using the contractibility function it follows that any two sufficiently close maps from a finite dimensional CW complex to $Y$ are homotopic (the required closeness might a priori depend on the dimension of the CW complex.) 
% With some work that can be improved to show that there is an $\eps>0$ such that any tow $\eps$ close maps into $Z\to Y$ are homotopic.  Here $Y$ is any topological space. (afaik, the latter is not written anywhere and we don't need this stronger statement ).

This easily implies that given a convergent sequence of Alexandrov spaces $X_i^m\to X_\infty$ with $\curv\ge \kappa, \diam\le D$ there are well defined maps $\pi_k(X_i)\to \pi_k(X_\infty)$ for large $i$ \cite{Per-collapsing-no-extr}. These are obtained by using a fine triangulation of a spheroid in $X_i$ and using contractibility radius to fill in the corresponding triangulation to get a spheroid in $X_i$, and the same with the homotopy of spheroids. 

In general the map  $\pi_k(X_i)\to \pi_k(X_\infty)$ need not be surjective as not every spheroid in $X_\infty$ can be lifted to $X_i$. One obvious exception is when $k=1$ since any loop in $X_\infty$ can be lifted to a loop in $X_i$ using a piecewise approximation by broken geodesics.

This gives

\begin{lem}\label{lem-conv-pi1}
Let $X^m_i\to X$ be a convergent sequence of Alexandrov spaces with with $\curv\ge \kappa, \diam\le D$.
Then for all large $i$ there is an epimorphism $\pi_1(X_i)\to \pi_1(X_\infty)$. In particular if all $X_i$ are simply connected then so is $X_\infty$.
\end{lem}
Note that the lemma is easily seen to fail for pointed convergence to a noncompact limit.

\begin{rem}
In general Lemma~\ref{lem-conv-pi1} works not just for Alexandrov spaces but in a much broader setting where one has a $\pi_1$ contractibility function for $X$, i.e. $X$ has the property that any loop contained in a ball of radius $r$ is contractible in a  ball of radius $\rho(r)$ where $\rho(r)\to 0$ as $r\to 0$.. In particular it applies to Ricci limits \cite[Theorem 1.2]{wang-ricci-limits}.
\end{rem}
Lemma \ref{lem-conv-pi1} in particular applies in the case we are interested in when all $X_i$ have $\curv \ge 1$ and nonempty boundary. Then all $X_i$ are contractible and in particular simply connected. hence $X_\infty$ is simply connected too. This rules out the possibility that $X_\infty=\RP^{m-1}$.

Now consider $X_\infty=\RP^k*\SS^l$ with $k+l=m-2$ and $k>0 , l\ge 0$.

We claim that in this case every spheroid in $X_\infty$ can be lifted to $X_i$ for all large $i$.

Let $f\co \SS^n\to X_\infty$ be a spheroid. We can choose a very fine triangulation of $f$ and try to lift simplices skeleta by skeleta. As was explained it's easy to construct a lifting on the 1-skeleton. To do an induction step we need to be able to fill a simplex $f_i\co \Delta^s\to  X_i$ provided a map on its boundary has already been constructed. This is easy to do near regular points in $X_\infty$ since the regular fiber is an interval and hence small balls in $X_i$ close to the balls around regular points in $X_\infty$ are contractible. However, this argument fails near singular points of $X_\infty$ and a different argument is required there. The singular set in $X$ is an isometric copy of $\SS^l$ which is extremal with normal spaces of directions isometric to $\RP^k$.

Let us assume for simplicity that $l=0$. Then $X_\infty$ is the spherical suspension $S(\RP^{m-2})$ over $\RP^{m-2}$ and there are only two isolated singular points - the vertices of the suspension. The general case is similar to this one but locally everything is crossed with a disk in $\R^l$.

Let $p$ be one of the vertices of $X_\infty=S(\RP^{m-2})$  and let $p_i\in X_i$ converge to $p$. 

{\bf Claim}: There is a small $r_0>0$ such that for any $r<r_0$ the points $p_i\to p$ can be chosen so that balls $B_r(p_i)$ are contractible for all large $i$. 

This obviously allows filling in simplexes in $B_r(p_i)$ and hence shows that $f$ can be lifted to a nearby spheroid $f_i\co S^n\to X_i$. Hence for all large $i$ the map $\pi_n(X_i)\to \pi_n(X_\infty)$ is onto. This leads to a contradiction since $X_i$ is contractible and $S(\RP^{m-2})$ is not. More specifically $\pi_2(S(\RP^{m-2}))\cong H_2(S(\RP^{m-2}))\cong \Z_2\ne 0$.

Thus it remains to prove the Claim above.

We can choose $r_0$ small enough so that for all $r<r_0$ the unit ball around the base point in $(X_\infty, \frac{1}{r}d_\infty, p)$ is  $o(1)$-close to  $B_1(o_p)$ in $T_pX_\infty\cong C(\RP^{m-2})$.
By the work of  Yamaguchi-Shioya \cite{Shi-Yam-collapsing}  and Yamaguchi \cite{yam-collapsing-ess} collapsing near $p$ can be understood as follows.

\begin{thm}\label{thm-local-collapse}

Let  $(X^m_i, \dd_i, q_i)\to (X_\infty, \dd_\infty, q_\infty)$ where all $X_i$  have  $\curv\ge \kappa$. Let $p\in X_\infty$.
Then for all sufficiently small $r$ there is $p_i\in X_i$ converging to $p$ such that one of the following two possibilities holds:

\begin{enumerate}
\item After passing to a subsequence $\dd_i(\cdot, p_i)$ has no critical points in $B_r(p_i)\setminus\{p_i\}$ or
\item There is a sequence $\delta_i\to 0$ such that 
\begin{enumerate}
\item For any $\lambda>1$ after for all sufficiently large $i$ the function $d(\cdot, p_i)$ has no critical points in the annulus $\{\lambda\delta_i\le d(x, p_i)\le r\}$
\item for any subsequential limit $(X_i, \frac{1}{\delta_i}\dd_i, p_i)\to (Y, \dd_Y, y_0)$ it holds that  $\dim Y\ge \dim X_\infty+1$.
\end{enumerate}
\end{enumerate}

\end{thm}
Note that we obviously have that in the above theorem $Y$ has $\curv\ge 0$.

We want to apply this theorem in our situation. If the first alternative in Theorem~\ref{thm-local-collapse} holds then as discussed we can lift any spheroid from $X_\infty$ to $X_i$ which eventually leads to a contradiction.
Let us examine what happens if the second alternative holds.
Since in our case $\dim X_\infty=m-1$ we must have that $\dim Y=m$. This means that the convergence  $(X_i, \frac{1}{\delta_i}\dd_i, p_i)\to (Y, \dd_Y, y_0)$ is noncollapsing.

It is easy to see that the ideal cone $Y(\infty)$ of $Y$ given by the limit $(Y, \frac{1}{\delta_i}\dd_Y, y_0)\to (Y(\infty), \dd_{Y(\infty)}, o_y)$  as $\delta_i \to 0$ has dimension at least $m-1$.
This implies that the soul of $Y$ has dimension 1 or 0. If it has dimension $0$ then $Y$ is contractible which by stability implies that all balls $\bar B_r(p_i)$ are contractible for large $i$ and we are again in the situation where we can lift spheroids.

Let us examine what happens if the soul $S$ of $Y$ is 1-dimensional. Then $S$ is $\mathbb{S}^1$ and by a standard argument using the splitting theorem we get that $Y$ is a flat bundle over $S_1$. More precisely, there is a nonnegatively curved Alexandrov space $Z^{m-1}$ and a point $z_o\in Z$ (the soul of $Z$)  such that $Y$ is isometric to $([0, 1]\times Z)/\sim$ where we glue $\{0\}\times Z$ to $\{1\}\times Z$ by some isometry $\phi\co Z\to Z$ which fixes $z_0$.
This in particular implies that for all large $R$ the spheres $S_R(y_0)$ have infinite fundamental groups. By the stability theorem and Morse theory this implies that for all large $i$ the spheres $S_r(p_i)$ have infinite fundamental groups as well.  But on the other hand the structure of collapsing away from $p$ can be understood by the fibration theorem since $M_\infty$ is smooth there. This gives that metric spheres $S_r(p_i)$ are homotopy equivalent  to the total space of an interval bundle over $S_r(p)$ which is homeomorphic to $ \RP^{m-1}$. This means that  $\pi_1(S_r(p_i))$ must be finite for all large $i$. This is a contradiction which shows that this case is impossible.

This means that  the case  $X_\infty=\RP^k*\SS^l$ with $k+l=m-2$ and $k>0 , l\ge 0$ cannot occur at all and the only possibility for $X_\infty$ is the round hemisphere $\SS^{m-1}_+$. 
\end{proof}

Next we consider the second part of Theorem \ref{thm-almost-rigidity}. 
The case where $\kappa > 0$ and $R = \pik/2$ follows easily from the first part of the Theorem since $\mathcal{H}^{m-1}(\partial B_{R,+}^{\kappa}) = \mathcal{H}^{m-1}(\partial B_{R}^{\kappa})$. We have
\begin{Thm}\label{thm-almost-rigid-4}
Given an integer $m$ and some $\kappa > 0$. Let $R = \pik/2$. For any $\eps > 0$, there exists $\delta(\kappa, m) > 0$ such that the following holds: Let $X^m$ be Alexandrov with $\curv \ge \kappa$. Assume that $\overline{B_R(p)}=X$ for some $p \in \partial X$ and that
 \begin{equation*}
	\mathcal{H}^{m-1}(\partial X) > \mathcal{H}^{m-1}(\partial B_{R, +}^{\kappa})-\delta.
\end{equation*}
Then there exists $Y$, which is either $L^{m, \kappa}_{\alpha}$ for some $\alpha \in (0, \pi/2]$ or the closed ball of radius $R$ in $M^{m-1}(\kappa)$, so that
\begin{equation*}
	\dd_{GH}(X, Y) < \eps.
\end{equation*}
\end{Thm}

For the last case, we have
\begin{Thm}\label{thm-almost-rigid-3}
Given an integer $m$, $\kappa \in \R$ and $R > 0$. If $\kappa > 0$, assume in addition that $R < \pik/2$. For any $\eps > 0$, there exists $\delta(\kappa, m, R) > 0$ such that the following holds: Let $X^m$ be Alexandrov with $\curv \ge \kappa$. Assume that $\overline{B_R(p)}=X$ for some $p \in \partial X$ and that
 \begin{equation*}
	\mathcal{H}^{m-1}(\partial X) > \mathcal{H}^{m-1}(\partial B_{R, +}^{\kappa})-\delta.
\end{equation*}
Then
\begin{equation*}
	\dd_{GH}(X, \overline{B_{R,+}^{\kappa}}) < \eps.
\end{equation*}
\end{Thm}

We give an outline of the proof.

Assume that we have a contradicting sequence $(X^m_i, \dd_i)$ of $\curv \ge \kappa$. This means that $\overline{B_R(p_i)}=X_i$ for some $p_i \in \partial X_i$ and $\mathcal{H}^{m-1}(\partial X_i) >\mathcal{H}^{m-1}(\partial B_{R,+}^{\kappa})-\delta_i$, where $\delta_i \to 0$ as $i \to \infty$, but $\dd_{GH}(X_i, \overline{B_{R,+}^{\kappa}}) > \epsilon$ for all $i$.  By compactness, we may take a convergent subsequence, which we also denote $(X^m_i, \dd_i)$. As before, if we can rule out collapse then by the rigidity theorem \ref{main thm} we will have a contradiction. 

Assume that $(X_i, \dd_i, p_i) \to (X_\infty, \dd_\infty, p_\infty)$ and the convergence is collapsing. Then by the same argument as in the proof of the first part of Theorem \ref{thm-almost-rigidity}, we have that $\dim X_\infty = m-1$ and $\partial X_i \to X_\infty$. Moreover, by the same fibration and covering argument we have that $\mathcal{H}^{m-1}(X_\infty)=\frac {1}{2} \mathcal{H}^{m-1}(\partial B^{\kappa}_{R,+})$. Since $\overline{B_R(p_\infty)} = X_\infty$, by absolute volume comparison in dimension $m-1$, we have that $\mathcal{H}^{m-1}(X_\infty) \le \mathcal{H}^{m-1}(B^{\kappa,m-1}_R)$, where $B^{\kappa, m-1}(R)$ is the $R$-ball in $M^{\kappa}(m-1)$. However, clearly $\frac{1}{2} \mathcal{H}^{m-1}(\partial B^{\kappa}_{R,+}) > \mathcal{H}^{m-1}(B^{\kappa,m-1}_R)$ which is a contradiction.
\begin{rem}
The same argument as in the above proof can be used to rule out collapsing in the proof of Theorem \ref{thm-almost-rigid} using that $\frac{1}{2} \mathcal{H}^{m-1}(\partial B^{\kappa}_{R}) > \mathcal{H}^{m-1}(B^{\kappa,m-1}_R)$  when $R<\pik/2$.
\end{rem}

\section{Boundary of Alexandrov spaces: background}
In this section we collect various known results about the boundaries of Alexandrov spaces as well as some folklore results. We also include several new results.

We try to present all the results in the order they need to be proved to avoid logical loops in the arguments. For the folklore results we present proofs.  We will stress topological results as they are less well-known and some are new. Throughout this section all Alexandrov spaces will be assumed to be finite dimensional.

Let $X^n$ be an $n$-dimensional Alexandrov space of $\curv \ge k$.
Recall that the boundary  $\partial X$ of $X$ is defined inductively as follows. For $n=1$ it's known that $X$ is a topological 1-manifold with (possibly empty) boundary and $\partial X$ is defined as the manifold boundary of $X$.
For $n>1$ the boundary is defined inductively: $p \in \partial X$ iff the space of directions $\Sigma_pX$ has boundary. This definition makes sense since $\Sigma_pX$ is an Alexandrov space of $\curv \ge 1$ of dimension $n-1$.
 Points which are not boundary points are called interior points. The set $X \setminus \partial X$ of all interior points is denoted by $\Intr X$.
 \begin{defn}
A metrizable space $X$ is called an MCS-space (space with multiple conic singularities) of dimension $n$ if  every point $x\in X$ has a neighborhood which is pointed homeomorphic to an open cone over a compact  $(n-1)$-dimensional MCS-space. Here we assume the empty set to be the unique $(-1)$-dimensional MCS-space.
\end{defn}

A compact $0$-dimensional MCS-space is a finite collection of points with discrete topology and a $1$-dimensional MCS-space is a graph.

By Perelman's Morse theory~\cite{Per-Morse}, an $n$-dimensional Alexandrov space $X$ is an $n$-dimensional MCS space. This also follows by induction on dimension from Perelman's stability theorem \cite{Per1, Ka07} which implies that a conical neighborhood of $x\in X$ is homeomorphic to $C(\Sigma_xX)$.

An open conical neighborhood of a point in an MCS-space is unique up to pointed homeomorphism ~\cite{Kwun}. As is true for MCS-spaces in general,  $X$ admits a canonical topological stratification.
We say that a point $p\in X$ belongs to the $l$-dimensional strata $X_l$ if $l$ is the maximal number $m$ such that the conical neighbourhood 
of $p$ is pointed homeomorphic to $\R^m\times K(S)$  for some  compact  MCS-space $S$. It is clear that $X_l$ is an $l$-dimensional topological manifold.

We will need the following key result of Perelman. We include its proof due to its importance.
\begin{thm}\label{thm-bry-top} \cite[Theorem 4.6]{Per1}
Let $\Sigma_1^{k_1}$ and $\Sigma_2^{k_2}$ be Alexandrov spaces of $\curv\ge 1$ such that $\R^{l_1}\times C(\Sigma_1)\overset{homeo}{\cong} \R^{l_2}\times C(\Sigma_2)$.
Then  $\Sigma_1$ has boundary iff $\Sigma_2$ does.
\end{thm}
\begin{proof}

Note that we necessarily have that $l_1+k_1=l_2+k_2$.
Let's assume $\Sigma_1$ has boundary. We need to show that $\Sigma_2$ does too.

Let us first prove the result for $k_1=1$. If $k_1=1$ then $\Sigma_1=I$ is a closed interval.
We proceed by induction on $k_2$. If $k_2=1$ then $\Sigma_2$ can only be $I$ or $S^1$ and since $\R^l\times C(I)$ is not homeomorphic to  $\R^l\times C(S^1)$ the theorem is clear in this case.

Induction step on $k_2$:
Let $k_2>1$ and suppose we have already proved the statement for smaller values of $\dim \Sigma_2$.

Suppose $F \co \R^{l_1}\times C(I)\to \R^{l_2}\times C(\Sigma_2^{k_2})$ is a homeomorphism and $\Sigma_2$ has no boundary.

Since $k_2>1=k_1$ we must have $l_1>l_2$. Therefore there is a point $x\in \R^{l_1}\times \{o_1\}$ whose image $y=F(x)$ does not lie in $\R^{l_2}\times \{o_2\}$, i.e., $y=(q, t, v_2)$ where $q\in \R^{l_2}, t>0, v_2\in \Sigma_2$. 
Since $\Sigma_2$ has no boundary $\Sigma_{v_2}(\Sigma_2)$ has no boundary either by definition.
Then a conical neighborhood of $y$ is homeomorphic to $\R^{l_2+1}\times C(\Sigma_{v_2}(\Sigma_2))$ by the stability theorem. By uniqueness of conical neighborhoods this implies that  $\R^{l_1}\times C(I)$ is homeomorphic to $\R^{l_2+1}\times C(\Sigma_{v_2}(\Sigma_2))$. This contradicts the induction assumption on $k_2$.

This proves that if $\R^{l_1}\times C(I)\overset{homeo}{\cong} \R^{l_2}\times C(\Sigma_2)$ then $\Sigma_2$ has boundary.

Let us now do induction on $k_1$ in the main statement of the theorem. We have just verified the base of induction $k_1=1$.

Let $k_1>1$ and suppose the theorem is proved for smaller $k_1$. Suppose  $F\co \R^{l_1}\times C(\Sigma_1^{k_1})\to \R^{l_2}\times C(\Sigma_2^{k_2})$ is a homeomorphism where $\Sigma_1$ has boundary but $\Sigma_2$ does not.

Let $v_1\in \partial \Sigma_1$. Let $x=(0, 1, v_1)\in  \R^{l_1}\times C(\Sigma_1)$ where $0\in \R^{l_1}, 1\in \R$. Let $y=F(x)$.   We have two possibilities.

{\bf Case 1} $y=q\times o_2$ lies on $R^{l_2}$ times the cone tip of $C(\Sigma_2)$.
Then a conical neighborhood of $x$ is homeomorphic to a conical neighborhood of $y$ which gives that $\R^{l_1+1}\times C(\Sigma_{v_1}(\Sigma_1))$ is homeomorphic to $\R^{l_2}\times C(\Sigma_2^{k_2})$.  Since $\Sigma_{v_1}(\Sigma_1)$ has boundary this contradicts the induction assumption.

{\bf Case 2} $y=(q,t,v_2)$ where $q\in \R^{l_2}, t>0, v_2\in\Sigma_2$. Then similarly to Case 1 we get that $\R^{l_1+1}\times C(\Sigma_{v_1}(\Sigma_1))$ is homeomorphic to $\R^{l_2+1}\times C(\Sigma_{v_2}(\Sigma_2))$,
which is again a contradiction to the induction assumption since $\Sigma_{v_1}\Sigma_1$ has boundary while  $\Sigma_{v_2}\Sigma_2$ does not.

\end{proof}

Theorem \ref{thm-bry-top} implies that boundary points are distinguished from interior points topologically and that a point in a conical neighborhood of an interior point cannot be a boundary point.  This can be further  sharpened (see Theorem~\ref{thm:boundary} below)  to say that boundary points have different local homology from interior points but in order to prove this the above result had to be proved first.

Theorem \ref{thm-bry-top} immediately  implies the following basic fact.
\begin{cor}\cite[Theorem 4.6]{Per1} \label{thm-bry-distinguished}
Let $X$ be Alexandrov. Then $\partial X$ is a closed subset of $X$.
\end{cor}

Theorem \ref{thm-bry-top}  also immediately implies that homeomorphisms between Alexandrov spaces preserve boundaries. Together with the stability theorem this implies that under noncollapsed convergence the boundaries converge to the  boundary in the limit.

\begin{thm}\label{bry-conv-to-bry}
Let $(X_i^m,p_i)\to (X^m,p)$ be a noncollapsing convergent sequence of Alexandrov spaces with $\curv\ge\kappa$. Let $d_i=\dd_i(p_i,\partial X_i)$ (if $\partial X_i=\varnothing$ we set $d_i=+\infty$). Then $X$ has nonempty boundary iff $d:=\liminf d_i<\infty$. Moreover, if $d<\infty$ then $(X_i^m,\partial X_i, p_i)\to (X^m,\partial X_i,p)$.
\end{thm}

Similar to Alexandrov spaces, for any $p\in \partial X$ there are two natural notions of the tangent space to the boundary $T_p\partial X$.
One is the cone over $\partial \Sigma_pX$.  Another is the limit of $(\partial X, p, \lambda d_X)$ as $\lambda\to \infty$.

Applying Theorem \ref{bry-conv-to-bry} to the convergence $(X, \lambda \dd, p) \to (T_pX, , \dd_{T_pX}, o_p)$ as $\lambda\to\infty$ gives that these two notions are the same. More precisely, we have
\begin{cor}\cite{Per1}\label{local-top-bry}
Let $X$ be a  Alexandrov space and let $p\in \partial X$. Then 
\[
(X, \partial X, \lambda \dd, p)\underset{\lambda\to\infty}{\longrightarrow} (T_pX=C(\Sigma_pX), \partial C(\Sigma_pX) =C(\partial \Sigma_pX), \dd_{T_pX}, o_p).
\]
Furthermore a small neighborhood of $p$ in $\partial X$ is homeomorphic  to $C(\Sigma_p\partial X)$.
\end{cor}
By induction on dimension in the last statement of the above theorem we get
\begin{cor}\label{cor-bry-MCS}
Let $X^m$ be an $m$-dimensional Alexandrov space with nonempty boundary.
Then $\partial X$ is an $(n-1)$-dimensional MCS-space.
\end{cor}

A key theorem in the theory of Alexandrov spaces is the following result of Perelman whose proof relies only on the stability theorem and the results stated above.
\begin{thm}[Doubling Theorem \cite{Per1}\label{thm-doubling}]
Let $X$ be an Alexandrov space of $\curv\ge \kappa$ with nonempty boundary. Then the doubling of $X$ along the boundary is an Alexandrov space of $\curv\ge \kappa$ without boundary.
\end{thm}
By Morse theory, Theorem~\ref{thm-bry-top} immediately implies the following theorem. 
\begin{thm}\cite{Pet-Per-extremal}\label{cor-bry-extremal}
For any Alexandrov space $X$ its boundary $\partial X$ is extremal in $X$.
\end{thm}

\begin{proof}
Let $p\notin \partial X$. Let $q\in \partial X$ be  such that $\dd(p,q)=\dd(p,\partial X)$. Let $f=d(\cdot, p)$.  By \cite[Theorem 4.1.2]{petrun-semiconcave}, we need only to show that $\nabla_q f=0$, i.e., $q$ is a critical point of $f$. Suppose not. Then by Morse theory $f$ is a local bundle map to $\R$ near $q$ and hence there exists $q'$ near $q$ with $f(q')<f(q)$ such that  a conical neighborhood of $q$ is homeomorphic to a conical neighborhood of $q'$. But this is impossible by  Theorem~\ref{thm-bry-top}  since $q'\notin \partial X$ and hence $\Sigma_{q'}X$ has no boundary.
\end{proof}

\begin{rem}
Theorem ~\ref{cor-bry-extremal} gives a different proof of Corollary \ref{cor-bry-MCS} using general results on topological structure of extremal subsets \cite{Pet-Per-extremal}.
\end{rem}

There is a third notion of the tangent space $T_p\partial X$ for $p\in \partial X$ where in the convergence of rescaled metrics on $\partial X$ one looks at intrinsic rather than extrinsic metrics.  This notion is generally different as should be expected. Nevertheless, one has the following.

\begin{thm}\label{intr-bdry-conv}
Let $X^m$ be Alexandrov and $p\in\partial X$. Then $(\partial X, \lambda \dd^{\partial}_X, p)\to (T_p\partial X, \dd^\partial_{T_p X}, o_p)$ as $\lambda\to\infty$, where $\dd^\partial$ stand for intrinsic metrics on the boundary on relevant spaces.
\end{thm}
\begin{proof}
This follows from Corollary~\ref{local-top-bry} along with the fact that that $\partial X$ is extremal and a result of Petrunin stating that given noncollapsing convergence $(X^n,F_i,p_i)\to (X^n,F,p)$ where $F_i$ are extremal in $X$, then the intrinsic metrics on $F_i$ also Gromov-Hausdorff converge to the intrinsic metric on $F$.
\end{proof}

Following Petrunin and Perelman \cite{Petr-Per-grad-curves} (see also \cite{petrun-semiconcave}) we can now adopt the following definition.

\begin{defn}\label{defn-semiconcave}
Let $X$ be an Alexandrov space. A locally Lipschitz function $f\co X\to \R$ is called \emph{semiconcave} if for any any point in $X$ has a neighborhood $U$ such that  its restriction of $f$ to any geodesic in $U$ is $\lambda$-concave for some $\lambda=\lambda(U)$. 

$f$ is called \emph{double semiconcave} if it is semiconcave in the case where $\partial X = \varnothing$. In the case where $\partial X \neq \varnothing$ then we also require that the canonical extension of $f$ to the double of $X$ is semiconcave in the above sense. 

\end{defn}

\begin{rem}
Let us comment on the last requirement in the above definition. Much of the theory of semiconcave functions including most of the theory of their gradient flows can be developed without that assumption (see \cite{AKPbook}).
However some results (such as Theorem~\ref{extr-flow-invariance} below)  require existence of supporting vectors (see \cite[Definition 1.3.6]{petrun-semiconcave}) to $f$ at all points.  Supporting vectors always exist at interior points of $X$  but may not exist at boundary points without the doubling assumption \cite{Petr-Per-grad-curves}. Let us note however that existence of supporting vectors for distance functions, i.e., to functions of the form $\dd(\cdot, A)$, outside $A$ can be proved directly irrespective of the existence of the boundary (see \cite[Proposition 2.2]{Per-Morse}). Therefore Theorem~\ref{extr-flow-invariance} for such functions can  be proved directly without any reference to the notion of a boundary.
\end{rem}
The following theorem of Petrunin and Perelman is very useful.
\begin{thm}\label{extr-flow-invariance} 
Let $X$ be an Alexandrov space and let $f\co X\to \R$ be double semiconcave.
Then the gradient flow of $f$ leaves any extremal subset of $X$ invariant.
\end{thm}
Since boundary is extremal by Corollary~\ref{cor-bry-extremal} this immediately yields

\begin{cor}\label{cor-bry-flow-inv}
Let $X$ be an Alexandrov space and let $f\co X\to \R$ be double semiconcave.
Then the gradient flow of $f$ leaves $\partial X$ invariant.
\end{cor}

Let us remark here that Theorem \ref{extr-flow-invariance}  and Corollary \ref{cor-bry-flow-inv} are false for semiconcave functions which are not double semiconcave. 
For example, $f=\dd(\cdot, \partial X)$ is known to be semiconcave along geodesics in $X$ \cite{Per1, Alexander-Bishop}, but its extension to the double of $X$ is not semiconcave, supporting vectors to $f$ do not exist at boundary points and the gradient flow of $f$ clearly does not leave $\partial X$ invariant.

The following result about the topology of the boundary is folklore but to the best of our knowledge the proof has not been carefully written. Therefore we record it here.
\begin{prop}\label{prop-bry-top}
Let $X^n$ be an Alexandrov space. Then
\begin{enumerate}
\item\label{item-mcs} $\partial X$ is an $(n-1)$-dimensional MCS space;

\item\label{item-bry-no-bry} $(\partial X)_{n-2}=\varnothing$;
 \item \label{homology-space-dir-bry}  For any $p\in \partial X$ it holds that $H_{n-2}(\Sigma_p\partial X,\Z_2)\ne 0$;
 \item\label{homology-bry}  Let $B$ be a compact connected component of  $\partial X$. Then $H_{n-1}(B,\Z_2)\ne 0$.
 
  \end{enumerate}
\end{prop}

\begin{proof}

Parts ~\eqref{item-mcs} and \eqref{item-bry-no-bry} follow from Corollary \ref{local-top-bry} by induction on dimension.

Using the fact that $(\partial X)_{n-2}=\varnothing$  parts \eqref{homology-space-dir-bry} and \eqref{homology-bry} follow by   \cite[Lemma 1]{GrPe-sphere-thm} which shows that  a compact MCS space with empty codimension 1 strata has nontrivial top homology with $\Z_2$ coefficients.

\end{proof}

\begin{rem}
The statement that  $(\partial X)_{n-2}=\varnothing$ can be interpreted as saying that the boundary of the boundary is empty.

\end{rem}

We can now prove the following equivalent descriptions of boundary points. Recall that $\mathcal S^k(X)$ is defined as the set of points $p\in X$ such that $T_pX$ isometrically splits off $\R^k$ but not $\R^{k+1}$.

\begin{thm}\label{thm:boundary}
Let $X^n$ be an Alexandrov space and let $p \in X$.
Then the following are equivalent
\begin{enumerate}
\item \label{item-boundary} $p\in \partial X$;
\item\label{item-contractible} $\Sigma_p(X)$ is contractible;
\item \label{item-space-dir-homology} $\tilde H_{n-1}(\Sigma_pX,\Z_2)= 0$;
\item \label{item-local-homology} $H_n(X,X\setminus \{p\},\Z_2)= 0$;
\item \label{item-closure-top-strata}  $p\in \bar X_{n-1}$;
\item \label{item-closure-geom-strata} $p\in \overline{  \mathcal S^{n-1}(X)}$;

\end{enumerate}
\end{thm}
All results in this theorem are contained in the literature but we collect all the arguments in one place and give a proof that avoids circular reasoning.

\begin{proof}[Proof of Theorem \ref{thm:boundary}]
In what follows all homology is with coefficients in $\Z_2$.
The theorem is trivial for $n=1$. Therefore from now on we assume that $n\ge 2$.

First observe that by the stability theorem applied to the convergence $(X, \lambda\dd_X, p)\underset{\lambda\to\infty}{\longrightarrow} (T_pX, \dd_{T_p X},o_p)$  we have that $H_n(X,X\setminus \{p\})\cong H_n(T_pX, T_pX\setminus\{o_p\})\cong H_{n-1}(\Sigma_p(X))$.

Let $\Int X =X\setminus \partial X$ be the interior of $X$. By induction on dimension using the stability theorem in the above convergence it follows that $(\Int X)_{n-1}=\varnothing$. Note that this means that $X_{n-1}\subset \partial X$.

Again using  \cite[Lemma 1]{GrPe-sphere-thm} this implies that for any interior point $p\in \Int X$ it holds that $H_n(X,X\setminus \{p\})\cong H_{n-1}(\Sigma_pX)\ne  0$.

If $p\in \partial X$ then $\Sigma_p$ is a space of $\curv\ge 1$ with boundary. Since the distance function to the boundary on $\Sigma_pX$ is strictly concave its soul must be a point and hence $\Sigma_pX$ is contractible and  $H_n(X,X\setminus \{p\})\cong H_{n-1}(\Sigma_p)= 0$.

This establishes $\eqref{item-boundary} \Leftrightarrow \eqref{item-contractible}\Leftrightarrow  \eqref{item-space-dir-homology} \Leftrightarrow  \eqref{item-local-homology} $.

Let us establish  $\eqref{item-boundary} \Leftrightarrow \eqref{item-closure-top-strata}  $.
We prove this by induction on $n$. The base case $n=1$ is clear.  Let $n>1$ and suppose the equivalence $\eqref{item-boundary} \Leftrightarrow \eqref{item-closure-top-strata}  $ has already been established for Alexandrov spaces of dimension less than  $n$. 
We have already established that $X_{n-1}\subset \partial X$. Since $\partial X$ is closed this implies that 
 $ \bar X_{n-1}\subset \partial X$.

Conversely, if $p\in\partial X$ then $\partial \Sigma_p\ne\varnothing$ and by the induction assumption $\Sigma_p$ has nontrivial $n-2$ strata.  Therefore the vertex $o$ lies in the closure of the $n-1$ strata of $C(X)$ and by the stability theorem $p\in \bar X_{n-1}$.

This proves that $\partial X=\bar X_{n-1}$ which establishes  $\eqref{item-boundary} \Leftrightarrow \eqref{item-closure-top-strata} $.

Let us establish   $\eqref{item-closure-geom-strata} \Leftrightarrow \eqref{item-boundary}  $.
By the splitting theorem it's obvious that if $p\in \mathcal S^{n-1}$ then $T_pX\cong \R^n_+$ and hence $p\in \partial X$. Since $\partial X$ is closed  it follows that $\overline{  \mathcal S^{n-1}(X)}\subset \partial X$.
The inverse inclusion  follows from ~\cite{Fuj-uniform-bounds-extr} since $\partial X$ is an extremal subset of $X$. Alternatively it follows from the fact that if $x\in X$ is regular and $q\in \partial X$ satisfies $\dd(x,q)=\dd(x,\partial X)$ then $q\in \mathcal S^{n-1}$. This holds for example because in the double of $X$ the point $q$ lies on the interior of a geodesic connecting $x$ with its mirror and hence $q$ must be a regular point in the double by \cite{Petr-parallel}.

%This equivalence can also be derived from ~\cite[12.7]{BGP}.
\end{proof}

By the Stability Theorem part \eqref{item-contractible} in Theorem \ref{thm:boundary} implies that the punctured conical neighborhood of a point in $\partial X$ is contractible. In a general $n$-dimensional $MCS$-space $X$ the conical neighborhood of a point in $X_{n-1}$ is homeomorphic to to $\R^{n-1}\times C(K)$ where $K$ is a finite set of points and $|K|\ne 2$.  If $X^n$ is an $n$-dimensional Alexandrov space then $X_{n-1}\subset \partial X$ and hence the above shows that  $|K|$ can only be equal to $1$ since if $n\ge 2$  and $|K|\ge 3$ the punctured conical neighborhood of the apex in $\R^{n-1}\times C(K)$ is not contractible.
In particular this means that $X^n$ cannot contain a point whose conical neighborhood is homeomorphic to $\R^{n-1}\times C(K)$ where $K$ is a finite set with $|K|\ge 3$.

Thus we get

\begin{cor}\label{cor-bry-conical}
Let $X$ be a $n$-dimensional Alexandrov space and $p\in X_{n-1}$. Then conical neighborhood of $p$ is homeomorphic to $\R^n_+$.
\end{cor}

For an $n$-dimensional Alexandrov space $X^n$ without boundary Grove and Petersen proved in~\cite{GrPe-sphere-thm} that the top topological strata is connected and furthermore if $X$ is compact then $H_n(X,\Z_2)\cong \Z_2$.  It is natural to ask if a similar result holds for connected components of the boundary.  Grove and Petersen's proof uses that the set of regular points in an Alexandrov space is convex. However this is not known for connected components of the boundary. It would follow from the conjecture that the boundary is also an Alexandrov space  but that conjecture remains wide open.

Nevertheless we show that the result of Grove and Petersen does hold for the boundaries of Alexandrov spaces. This sharpens  Proposition ~\ref{prop-bry-top}.

\begin{prop}\label{prop-bry-top-2}
Let $X^n$ be an $n$-dimensional Alexandrov space with $n\ge 2$.
Then for every connected component  $Y$ of $\partial X$ it holds that $Y_{n-1}$ is connected and if $Y$ is compact then $H_{n-1}(Y)\cong \Z_2$.
\end{prop}

For the proof we need the following.

\begin{lem}\label{lem-bry-pos-connected}
Let $X^n$ be Alexandrov of $\curv\ge 1$ such that $n\ge 2$ and $\partial X\ne \varnothing$.

Then $\partial X$ is connected.
\end{lem}

\begin{proof}

Suppose $\partial X$ is not connected. Let $p$ be the soul of $\Sigma$.  By  \cite{Petr-Per-grad-curves}[Theorem 1.1]   we have that $\Sigma_p X$ is homeomorphic to $\partial X$. Therefore $\Sigma_pX$ is not connected. Hence $T_pX$ has more than one end and contains a line. Therefore it splits a line isometrically by the splitting theorem and since $n\ge 2$ this implies that $\Sigma_pX$ is a spherical suspension over a nonempty set and hence is connected.

An alternative proof can be given by looking at the double $Y$ of $X$ along the boundary of $X$. If $\partial X$ has more than one component then it is easy to see using that $X$ is contractible that $Y$ has infinite $\pi_1$. But $Y$ is Alexandrov of $\curv\ge 1$ and hence must have finite $\pi_1$.
\end{proof}
\begin{proof}[Proof of Proposition~\ref{prop-bry-top-2}]

It is enough to prove that $Y_{n-1}$ is connected. Then for compact $Y$  the isomorphism $H_{n-1}(Y)\cong \Z_2$ follows by  \cite[Lemma 1]{GrPe-sphere-thm} using that $Y_{n-2}=\varnothing$.

We proceed by induction. The base of induction $n=2$ is straightforward. 

Induction step. Let $p\in Y$. By Corollary~\ref{local-top-bry} a conical neighborhood $U$ of $p$ in $Y$ is homeomorphic  to $C(\partial \Sigma_pX)$.

By Lemma \ref {lem-bry-pos-connected} $\partial \Sigma_pX$ is connected and by the induction assumption its top strata is connected. Therefore the same holds for $U$, that is, $U_{n-1}$ is connected. Now the result trivially follows from connectedness of $Y$.
\end{proof}

\bibliographystyle{alpha}
\bibliography{biblio}

\end{document}